\documentclass[smallextended]{svjour3}       

\usepackage{a4wide}
\usepackage{graphicx}
\usepackage{amstext, amsmath, amssymb}
\usepackage[shadow]{todonotes}
\usepackage{booktabs}
\usepackage{url}
\usepackage{import}
\usepackage{amsfonts}
\usepackage{amsbsy}
\usepackage{fancyvrb}
\usepackage{stmaryrd}
\usepackage{pdfsync}
\usepackage{graphicx}
\usepackage{algorithm}
\usepackage[shadow]{todonotes}
\usepackage[numbers]{natbib}

\numberwithin{equation}{section}
\numberwithin{figure}{section}
\numberwithin{table}{section}

\numberwithin{remark}{section}
\numberwithin{proposition}{section}
\numberwithin{lemma}{section}
\numberwithin{theorem}{section}

\smartqed

\newcommand{\R}{\mathbb{R}}
\newcommand{\N}{\mathbb{N}}
\newcommand{\foralls}{\forall\,}
\newcommand{\dx}{\,\mathrm{d}x}

\DeclareMathOperator{\Div}{div}

\DeclareMathOperator{\Kern}{ker}

\DeclareMathOperator{\diam}{diam}

\newcommand{\mesh}{\mathcal{T}}
\newcommand{\jump}[1]{[#1]}

\newcommand{\apriori}{\emph{a~priori}}

\DefineVerbatimEnvironment{code}{Verbatim}{frame=single,rulecolor=\color{blue}}

\newcommand{\OmegaO}{\Omega_{_\mathrm{O}}}

\newcommand{\bfu}{\boldsymbol{u}}
\newcommand{\bff}{\boldsymbol{f}}
\newcommand{\bfg}{\boldsymbol{g}}
\newcommand{\bfv}{\boldsymbol{v}}

\newcommand{\bfn}{\boldsymbol{n}}

\newcommand{\Oast}{\Omega^{\ast}}

\newcommand{\meshast}{\mathcal{T}^{\ast}}

\newcommand{\tnast}{\tn_{\ast}}

\newcommand{\mcT}{\mathcal{T}}
\newcommand{\mcV}{\mathcal{V}}
\newcommand{\mcE}{\mathcal{E}}

\newcommand{\mcA}{\mathcal{A}}
\newcommand{\tn}{|\mspace{-1mu}|\mspace{-1mu}|}

\newcommand{\ndot}{\nabla \cdot}
\newcommand{\bfw}{\boldsymbol{w}}

\newcommand{\nablan}{\partial_{\bfn}}
\newcommand{\OmcupOm}{\Omega_1 \cup \Omega_2}
\newcommand{\mcupm}{{\mathcal{T}^{\ast}_1} \cup \mesh_2}
\newcommand{\meanvalue}[1]{\langle #1 \rangle}
\newcommand{\ifnormalpha}[1]{\ifnorm{#1}{\alpha}}
\newcommand{\ifnorm}[2]{\| #1 \|_{#2,h,\Gamma}}

\newcommand{\bfphi}{\boldsymbol{\phi}}
\newcommand{\picorr}{\pi^c}

\begin{document}

\title{\bf A stabilized Nitsche overlapping mesh method for the Stokes problem}
\author{Andr\'e Massing \and
        Mats G.\ Larson \and
      Anders Logg \and
        Marie E.\ Rognes}
\institute{Andr\'e Massing 
          \at 
          Simula Research Laboratory, Oslo, Norway \\
          Tel.: +47 46 95 74 01\\
          Fax:  +47 67 82 82 01\\
          \email{massing@simula.no}           
          \and
          Mats G.\ Larson \at
          Department of Mathematics, Ume{\aa} University, Ume{\aa}, Sweden
          \and
          Anders Logg 
          \at 
          Simula Research Laboratory, Oslo, Norway 
          \and
          Marie E.\ Rognes
          \at 
          Simula Research Laboratory, Oslo, Norway
        }

\date{Received: \today / Accepted: }
\maketitle

\begin{abstract}
  We develop a Nitsche-based formulation for a general class of
  stabilized finite element methods for the Stokes problem posed on a
  pair of overlapping, non-matching meshes. By extending the
  least-squares stabilization to the overlap region, we prove that the
  method is stable, consistent, and optimally convergent. To avoid an
  ill-conditioned linear algebra system, the scheme is augmented by a
  least-squares term measuring the discontinuity of the solution in
  the overlap region of the two meshes. As a consequence, we may prove
  an estimate for the condition number of the resulting stiffness
  matrix that is independent of the location of the
  interface. Finally, we present numerical examples in three spatial
  dimensions illustrating and confirming the theoretical results.
\end{abstract}

\section{Introduction}

Overlapping mesh methods offer many advantages over standard finite
element methods that require the generation of a single conforming
mesh resolving the full computational domain. With overlapping mesh
methods, the computational domain may instead be described by a set of
overlapping and non-matching meshes. In particular, different
subdomains may be meshed independently and then collected to form the
full domain. This feature is particularly useful in engineering
applications where meshes for physical components may be reused in
different configurations. Another important example is the simulation
of the flow around a complex object embedded in a channel. One may
then create a mesh that discretizes a fixed and simple domain such as
a cube or a sphere surrounding the complex object. This mesh may then
be imposed on top of a fixed background mesh for the simulation of the
flow around the object inserted at different locations in a domain of
interest. A particular advantage of this approach is that it allows
the creation of a fixed graded mesh to resolve boundary layers close
to the surface of the complex object. This is illustrated in
Figure~\ref{p3:fig:stokes_overlapping} for a simple two-dimensional
airfoil embedded in a channel.

In this work, we introduce an overlapping mesh method for Stokes flow
with constant viscosity across the artificial mesh interface.  The
Stokes problem reads: find the velocity
$\bfu: \Omega \subset \R^d \rightarrow
\R^d$ and the pressure $p:\Omega \to \R$ such that
\begin{subequations} \label{p3:eq:strongform}
  \begin{alignat}{3}
    - \Delta \bfu + \nabla p &= \bff & \quad & \text{in $\Omega$},
    \label{p3:eq:strong-stress}
    \\
    \nabla \cdot \bfu &=0 & \quad & \text{in $\Omega$},
    \label{p3:eq:strong-divergence}
    \\
    \bfu&=\bfg & \quad & \text{on $\partial \Omega$},
    \label{p3:eq:strong-dirichlet}
  \end{alignat}
\end{subequations}
where $\Omega$ denotes a bounded domain in $\R^d$, $d=2$ or $3$, with
Lipschitz boundary $\partial \Omega$, and where $\bff\in L^2(\Omega)$
and $\bfg \in H^{1/2}(\partial \Omega)$ are given functions. To
satisfy~\eqref{p3:eq:strong-divergence}, we assume that the mean value
of $\bfg \cdot \bfn$ vanishes; $\bfn$ denoting the outward unit normal
to $\partial \Omega$.

\begin{figure}
  \begin{center}
    \includegraphics[width=0.5\textwidth]{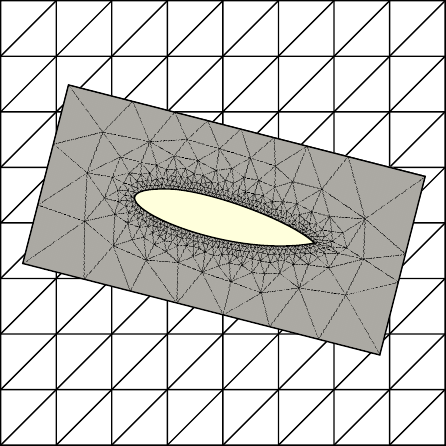}
    \caption{The stabilized Nitsche overlapping mesh method
      presented in this work allows the simulation of Stokes flow
      around a complex object (here a simple two-dimensional airfoil)
      described by a matching mesh of its surroundings imposed on top
    of a non-matching fixed background mesh.}
    \label{p3:fig:stokes_overlapping}
  \end{center}
\end{figure}

Our formulation is based on a general stabilized Galerkin finite
element method for the Stokes problem and enforcement of the interface
conditions via Nitsche's method. In order to prove stability, we let
the least-squares stabilization terms extend to all elements that
intersect the computational domain. As a result, these terms will
appear twice in any overlap regions. In addition, we include a certain
least-squares term penalizing the difference between the velocity
solutions in the overlap region. This allows us to prove stability and
optimal order error estimates as well as to control the condition
number of the resulting algebraic problem.

The method proposed here can be viewed as an extension to the Stokes
problem of earlier work by~\citet{HansboHansboLarson2003} who
developed a Nitsche overlapping mesh method for a second order
elliptic model problem. Also,~\citet{BeckerBurmanHansbo2009} presented
a Nitsche extended finite element method for incompressible elasticity
based on low-order ($[P_1^c]^d \times P_0$) elements. Moreover, the
least-squares penalty of the velocity differences is related to the
mesh tying approach proposed by~\citet{DayBochev2008}, who formulate a
least-squares problem for a system consisting of the partial
differential equation together with the interface conditions. Note
that in our method, the interface conditions are enforced using
Nitsche's method, while the least-squares terms on the overlap are
only included to prove the stability of the method and to control the
condition number.  In a related work~\citep{Massing2012c}, we present
a stabilized Nitsche fictitious domain method for the Stokes problem.

The implementation of the overlapping mesh method in three space
dimensions is a challenging problem. A realization of the method
proposed in this work entails computing the intersection of
arbitrarily superimposed tetrahedral meshes and integration over
arbitrarily cut tetrahedra. Such a realization has been developed as
part of the C++ library \rm{DOLFIN-OLM}
(\url{http://launchpad.net/dolfin-olm}) extending the FEniCS Project
software~\citep{LoggMardalEtAl2011,LoggWells2010a,Logg2007,KirbyLogg2006,LoggOlgaardEtAl2012a,Alnaes2011a,AlnaesLoggEtAl2009a}. For
a discussion of the computational aspects, we refer to our previous
work~\citep{Massing2012a} and the related paper~\citep{Massing2012c}.

The remainder of this work is organized as follows. We first summarize
our assumptions and notation in Section~\ref{p3:sec:preliminaries}.
The overlapping mesh method is then formulated in
Section~\ref{p3:sec:stokes-olm-method}.
Sections~\ref{p3:sec:approx-est}--\ref{p3:sec:a-priori} are devoted to
the stability and \apriori{} error analysis of the proposed method,
while the condition number estimate is presented in
Section~\ref{p3:sec:condition-number}. Finally, we demonstrate the
proposed method for a sample application in
Section~\ref{p3:sec:num-examples}, and present numerical convergence
results and condition number estimates to support our theoretical
results.

\section{Preliminaries}
\label{p3:sec:preliminaries}

In this section, we review the notation used throughout the remainder
of this work. We also summarize a standard stabilized Stokes
formulation to lay the foundations for the formulation of the
overlapping mesh method in Section~\ref{p3:sec:stokes-olm-method}.

\subsection{Finite element spaces}
\label{p3:ssec:notation}

In what follows, $H^s(\Omega)$ denotes the standard Sobolev space of
order $s \in \N$, defined on an open and bound domain $\Omega$ with
Lipschitz boundary $\partial \Omega$. We write
$(\cdot,\cdot)_{s,\Omega}$, $ \| \cdot \|_{s,\Omega}$ and $| \cdot
|_{s,\Omega}$ for the inner product, norm and semi-norm on
$H^s(\Omega)$, respectively. The index $s$ will be dropped when $s =
0$.

For a given, shape-regular tessellation $\mathcal{T}$ of $\Omega$, we
let the associated discrete velocity space $V_h \subset
[H^1(\Omega)]^d$ be the space of continuous piecewise polynomial
$\R^d$-valued vector fields of order $k$, and let the pressure space
$Q_h \subset L^2(\Omega) \cap C^0(\Omega)$ consist of continuous
piecewise polynomials of order $l$.  To emphasis the order of the
underlying polynomials, we occasionally write $V_h^k$ and $Q_h^l$.

\subsection{Stabilized Stokes elements}
\label{p3:ssec:stabilized-stokes-fem}

We recall the definition of consistently stabilized finite element
methods for the Stokes problem,
following~\citet{Franca1993,Barth2004}. We first define the bilinear
and linear forms
\begin{align}
  a_h(\bfu_h,\bfv_h) &= (\nabla \bfu_h , \nabla{\bfv_h})_{\Omega},
  \\
  b_h(\bfv_h,p_h) &=  - (\ndot \bfv_h, p_h)_{\Omega},
  \\
  l_h(\bfv_h) &= (\bff,\bfv_h)_{\Omega}.
\end{align}
As categorized by \citet{Barth2004}, a wide class of consistently
stabilized mixed finite element formulation for the Stokes problem can
be recast in the following form: find $(\bfu_h, p_h) \in V_h \times
Q_h$ such that
\begin{equation*}
  A_h(\bfu_h,p_h;\bfv_h,q_h) = L_h(\bfv_h,q_h)
  \quad \foralls (\bfv_h, q_h) \in V_h \times Q_h,
\end{equation*}
where
\begin{align*}
  A_h(\bfu_h,p_h;\bfv_h,q_h) &=
  a_h(\bfu_h,\bfv_h) +
  b_h(\bfv_h,p_h)  +
  b_h(\bfu_h,q_h)  -
  S_h(\bfu_h,p_h;\bfv_h,q_h),
  \\
  L_h(\bfv,q) &= l_h(\bfv)
  -\delta \sum_{T\in \mcT}
  h_T^2 (\bff,-\alpha\Delta\bfv_h + \beta\nabla q_h)_T,
\end{align*}
and where the stabilization is given by
\begin{align*}
  S_h(\bfu_h,p_h;\bfv_h,q_h)
  &= \delta \sum_{T \in \mcT} h_T^2
  (-\Delta \bfu_h + \nabla p_h,-\alpha \Delta \bfv_h + \beta \nabla
  q_h)_T.
\end{align*}
The parameters $\alpha$ and $\beta$ are chosen from the sets
$\{-1,0,1\}$ and $\{-1,1\}$, respectively and $\delta$ denotes some
positive constant. \citet{Barth2004} point out that the choice
$(\alpha,\beta) = (1,1)$ corresponds to the classical scheme
introduced by~\citet{HughesFrancaEtAl1986}, while on the other hand,
the method by~\citet{Douglas1989} can be constructed from the
parameter choice $(\alpha,\beta) = (-1,1)$.  In what follows, we will
focus on these two families of stabilized methods for the Stokes
problem.

\subsection{A domain decomposition model problem for the Stokes problem}
\label{p3:ssec:dd-model-problem}

Let $\Omega = (\overline{\Omega_1 \cup \Omega_2})^{\circ}$ be a domain
in $\R^d$ with Lipschitz boundary $\partial \Omega$, consisting of two
(open and bounded) disjoint subdomains $\Omega_1$ and $\Omega_2$
separated by the interface $\Gamma = \partial \overline{\Omega}_1 \cap
\partial \overline{\Omega}_2$.  To develop a Nitsche based overlapping
mesh method for the Stokes problem, we consider the following domain
decomposition model problem for \eqref{p3:eq:strongform}: find $\bfu:
\Omega \rightarrow \R^3$ and $p: \Omega \rightarrow \R$ such that
\begin{alignat}{3}
  -\Delta \bfu_i  + \nabla p_i  &= \bff_i &\quad &\text{in }  \Omega_i,
  \quad i = 1,2,
  \label{p3:eq:strong_stokes_equation_stress_form}
  \\
  \label{p3:eq:strong_stokes_equation_incompress_cond}
  \ndot \bfu_i &= 0 &\quad &\text{in } \Omega_i, \quad i = 1,2,
  \\
  \jump{\bfu} &= 0 &\quad &\text{on } \Gamma,
  \label{p3:eq:strong_cont_condition}
  \\
  \jump{\nablan \bfu - p \bfn} &= 0 &\quad &\text{on } \Gamma,
  \label{p3:eq:strong_flux_condition}
  \\
  \bfu &= 0 &\quad &\text{on } \partial \Omega.
  \label{p3:eq:dirichlet_bc}
\end{alignat}
Here and in the following, $v_i = v|_{\Omega_i}$ denotes the
restriction of a function or vector field $v$ to a subdomain
$\Omega_i$. Furthermore, $\bfn$ is the unit normal of $\Gamma$
directed from $\Omega_2$ into $\Omega_1$, $\nablan \bfu \equiv \bfn
\cdot \nabla \bfu$, and $\jump{v} = v_2 - v_1$ denotes the jump in a
function over the interface $\Gamma$.
\begin{figure}[htbp]
  \begin{center}
    \includegraphics[width=0.40\textwidth]{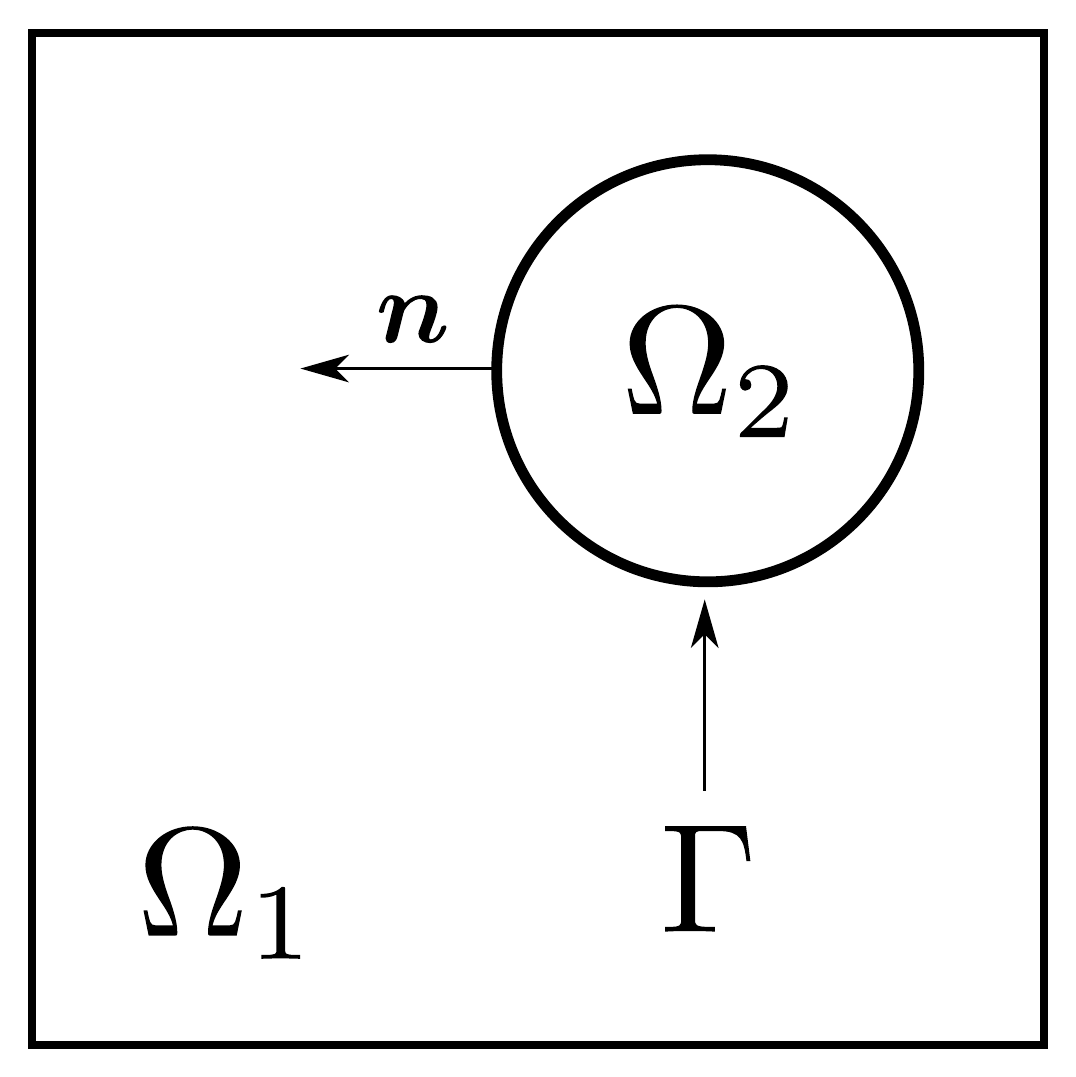}
    \caption{Decomposition of the domain $\Omega$ by introducing the
      artificial interface $\Gamma$. The weak coupling along the
      interface $\Gamma$ by the Nitsche method allows independent
      meshes for the subdomains $\Omega_1$ and $\Omega_2$, including
      overlapping meshes.}
    \label{p3:fig:model_domain}
  \end{center}
\end{figure}

The decomposition of $\Omega$ into the two
subdomains $\Omega_1$ and $\Omega_2$ motivates the introduction of the
broken Sobolev spaces
\begin{equation}
  H^s(\OmcupOm) = H^s(\Omega_1)
  \bigoplus H^s(\Omega_2),  \quad s\geqslant 0,
  \label{p3:eq:broken-space}
\end{equation}
endowed with the norm
\begin{equation*}
  \| v \|_{s,\OmcupOm}^2 = \| v_1 \|_{s, \Omega_1}^2 + \| v_2
  \|_{s,\Omega_2}^2.
\end{equation*}

The key idea in developing in Nitsche-type methods is to now replace
the strong continuity conditions~\eqref{p3:eq:strong_cont_condition}
and~\eqref{p3:eq:strong_flux_condition} by a weak formulation
\citep{Nitsche1971,HansboHermansson2003,HansboHansbo2002,HansboHansboLarson2003,Hansbo2005}.
This approach is analogous to that of discontinuous Galerkin methods
for elliptic equations~\citep{ArnoldBrezziCockburnEtAl2002}.  Starting
with suitable finite element discretizations of $H^s(\Omega_1)$ and
$H^s(\Omega_2)$, a weak formulation can be obtained by multiplying
with test functions, integrating by parts and adding certain
symmetrization and stabilization terms.  A typical example of the
resulting interface form will be given as part of the method we
present in Section~\ref{p3:sec:stokes-olm-method}.

We remark that the introduction of the interface $\Gamma$ is purely
artificial in our application case and solely serves the purpose of
decomposing the domain into suitable subdomains to ease and decouple
the subsequent meshing process.

\subsection{Overlapping meshes}
\label{p3:ssec:overlapping-meshes}

We consider a situation where a background mesh $\mesh_0$ is given for
$\Omega = (\overline{\Omega_1 \cup \Omega_2})^{\circ}$ and another
mesh $\mesh_2$ is given for the overlapping domain $\Omega_2$ (see
Figure~\ref{p3:fig:model_domain}). Both meshes are assumed to consist
of shape-regular simplices $T$.  We note that the tessellation
$\mesh_0$ of the background domain $\Omega$ may be decomposed into
three disjoint subsets:
\begin{equation}
  \label{p3:eq:mesh-splitting}
  \mesh_0 = \mesh_{0,1} \cup \mesh_{0,2} \cup \mesh_{0,\Gamma},
\end{equation}
where $\mesh_{0,1} = \{ T \in \mesh_0 : T \subset \overline{\Omega}_1
\} $, $\mesh_{0,2} = \{ T \in \mesh_0 : T \subset \overline{\Omega}_2
\} $ and $\mesh_{0,\Gamma} = \{ T \in \mesh_0 : | T \cap \Omega_i | >
0, \ i=1,2 \} $ denote the sets of \emph{not}, \emph{completely} and
\emph{partially} overlapped elements relative to $\Omega_2$,
respectively. The meshes $\meshast_1$ and $\mesh_1$ are then defined
by
\begin{align}
  \meshast_1 &= \mesh_{0,1} \cup \mesh_{0,\Gamma},
  \label{p3:eq:meshast-1}
  \\
  \mesh_1 &= \{ T \cap \overline{\Omega}_1 : T \in \meshast_1 \}.
  \label{p3:eq:mesh-1}
\end{align}
Moreover, we introduce the tessellated domain $\Oast_1 = \bigcup_{T
  \in \meshast_1} T$ and the overlap region $ \OmegaO = \Omega_2 \cap
\Oast_1$.  To ease the notation, we occasionally refer to $\mesh_2$ as
$\meshast_2$ and to $\Omega_2$ as $\Oast_2$.  See
Figure~\ref{p3:fig:overlapping_notation} for an illustration of the
various mesh parts and regions.  Furthermore, we introduce the
notation $\partial_e \meshast_j$ (the \emph{exterior facets}) and
$\partial_i \meshast_j$, $j=1,2$ (the \emph{interior facets}), for the
set of facets which belong to either one or two elements,
respectively. Here, a facet means an edge in $\R^2$ and face in
$\R^3$. In accordance with~\eqref{p3:eq:mesh-1}, the sets $\partial_i
\mesh_1$ and $\partial_i \mesh_1$ denote the corresponding set of
intersected facets $F \cap \overline{\Omega}_1$.
\begin{figure}
  \begin{center}
    \def\svgwidth{0.71\textwidth}
    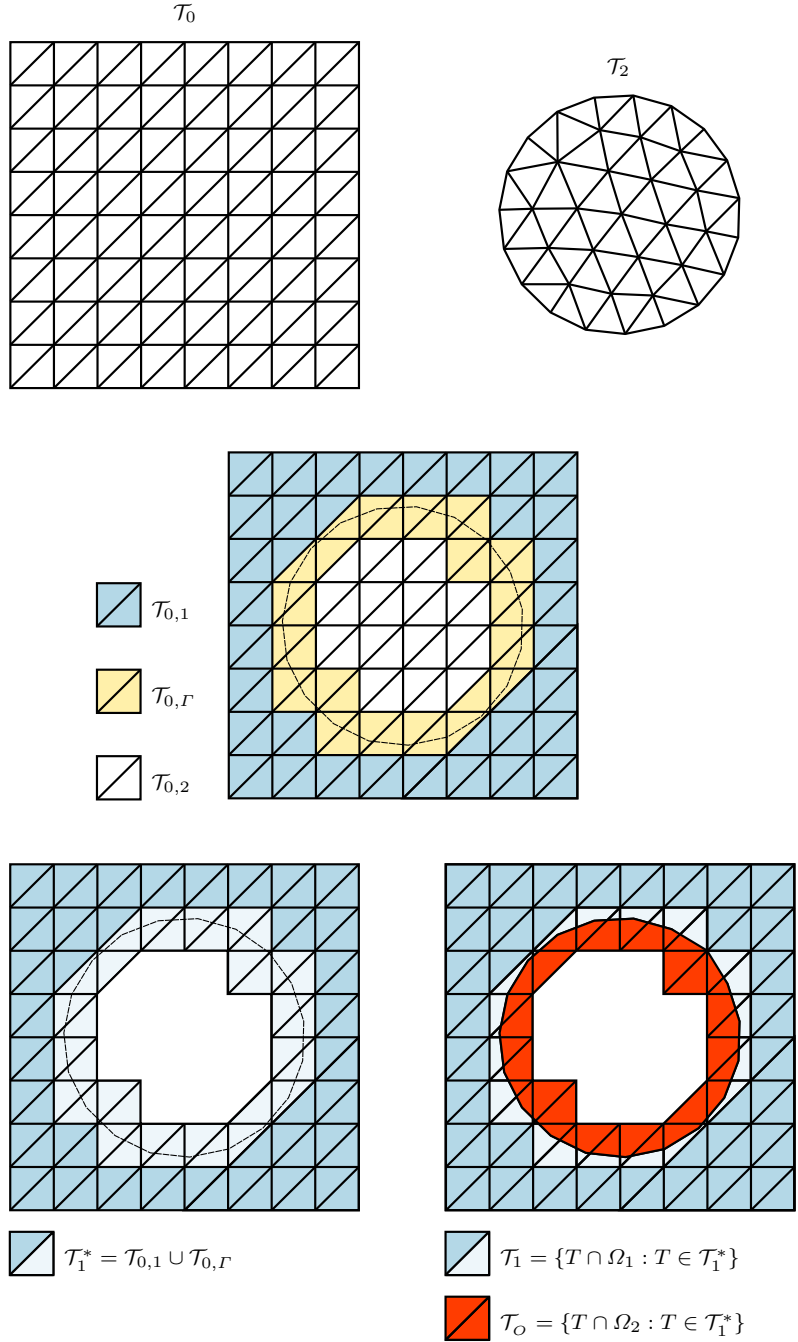
    \caption{Summary of notation for overlapping meshes. The starting
      point is a background mesh $\mesh_0$ and an overlapping mesh
      $\mesh_2$. The background mesh $\mesh_0$ is then partitioned
      based on its intersection with the boundary $\Gamma$ of the
      overlapping domain.}
    \label{p3:fig:overlapping_notation}
  \end{center}
\end{figure}

In addition to the shape-regularity, we require that the mesh sizes
are compatible over the interface $\Gamma$. More precisely, we assume
that there exist a mesh-independent constant $C > 0$ such that for all
$T \in \mesh_0$ and all $\widetilde{T} \in \mesh_2$ it holds
\begin{equation}
  C^{-1} h_{T} \leqslant  h_{\widetilde{T}}  \leqslant C h_{T}
  \label{p3:eq:meshsize-combatibility}
\end{equation}
whenever $T \cap \widetilde{T} \cap \Gamma \neq \emptyset$.
Similar assumptions were made in \citet{HansboHansboLarson2003}.

Finally, we compose suitable finite element spaces for the overlapping
meshes. For each mesh $\meshast_i$, $i=1,2$, let $\mcV_{h,i}$ be a finite element
space of continuous fixed-order polynomials on $\meshast_i$.
Similar to the broken Sobolev space \eqref{p3:eq:broken-space},
we introduce the composite finite element space for the whole domain
$\Omega$
\begin{equation}
  \mcV_h = \mcV_{h,1} \bigoplus \mcV_{h,2},
  \label{p3:eq:broken-fem-space}
\end{equation}
and define the broken Sobolev space norm as before by
\begin{equation}
  \label{p3:eq:broken}
  \| v_h \|_{s,\mcupm}^2 = \| v_{h,1} \|_{s,\meshast_1}^2 + \| v_{h,2}
  \|_{s,\mesh_2}^2.
\end{equation}
for $v_h \in \mcV_h$. Note that both $v_{h,1}$ and $v_{h,2}$
contribute to the norm in the overlap region $\OmegaO$.  In
particular, we define the composite finite element spaces $V_h$ and
$Q_h$ for the velocity and pressure, respectively, by
\begin{equation}
  \label{p3:eq:elementspaces}
  V_h = V_{h,1}^k \bigoplus V_{h,2}^k, \quad
  Q_h = Q_{h,1}^l \bigoplus Q_{h,2}^l
\end{equation}

\section{A Nitsche overlapping mesh method for the Stokes problem}
\label{p3:sec:stokes-olm-method}
Given a domain $\Omega$, overlapping meshes $\meshast_1$ and $\mesh_2$
and the composite finite element spaces $V_h$ and $Q_h$ as introduced
in~\eqref{p3:eq:elementspaces}, we define the bilinear form $A_h$ by
\begin{equation}
  \begin{split}
  \label{p3:eq:Ah-olm}
  A_h(\bfu_h,p_h;\bfv_h,q_h)
  = a_h(\bfu_h,\bfv_h) +
  b_h(\bfv_h,p_h)  +
  b_h(\bfu_h,q_h)  \\
  + s_h(\bfu_h,\bfv_h)
  - S_h(\bfu_h,p_h;\bfv_h,q_h) ,
  \end{split}
\end{equation}
where the bilinear forms $a_h$, $b_h$, $s_h$ and $S_h$
are given by
\begin{align}
  \label{p3:eq:ah-olm}
  a_h(\bfu_h,\bfv_h)
  &=
  (\nabla \bfu_h , \nabla{\bfv_h})_{\Omega_1 \cup \Omega_2}
  - {
    (\meanvalue{\nablan \bfu_h},
    \jump{{\bfv_h}})_{\Gamma}
  }
  \\
  &\phantom{=}\;
  - (\meanvalue{ \nablan \bfv_h}, \jump{{\bfu_h}}  )_{\Gamma}
  + \gamma (h^{-1} \jump{{\bfu_h}}, \jump{{\bfv_h}} )_{\Gamma},
  \nonumber
  \\
  \label{p3:eq:bh-olm}
  b_h(\bfv_h,q_h)
  &=  - (\ndot \bfv_h, q_h)_{\Omega_1 \cup \Omega_2} +
  (\bfn \cdot \jump{\bfv_h},\meanvalue{q_h})_{\Gamma},
  \\
  \label{p3:eq:sh-olm}
  s_h(\bfu_h,\bfv_h)
  &=
  (\nabla(\bfu_{h,1} - \bfu_{h,2}), \nabla(\bfv_{h,1} - \bfv_{h,1}))_{\OmegaO},
  \\
  \label{p3:eq:Sh-olm}
  S_h(\bfu_h,p_h;\bfv_h,q_h)
  &= \delta \sum_{T \in \mcupm } h_T^2
  (-\Delta \bfu_h + \nabla p_h,-\alpha \Delta \bfv_h + \beta \nabla
  q_h)_T .
\end{align}
Here, $\meanvalue{v}$ denotes a weighted average $\meanvalue{v} =
\alpha_1 v_1 + \alpha v_2$ with $\alpha_1 + \alpha_2 =1$ for a field
$v$ defined on $\Omega$ and discontinuous across the interface
$\Gamma$.  In accordance with~\citet{HansboHansboLarson2003} and to
simplify the presentation, we choose $\meanvalue{v} = v_2$.  Finally,
the linear form $L_h$ is defined by
\begin{equation}
  L_h(\bfv,q) = (\bff,\bfv)
  -\delta \sum_{T\in \mcupm}
  h_T^2 (\bff,-\alpha\Delta\bfv_h + \beta\nabla q_h)_T.
  \label{p3:eq:Lh-olm}
\end{equation}
With these definitions, the Nitsche based overlapping mesh method for
the Stokes problem~\eqref{p3:eq:strongform} reads: find $(\bfu_h,p_h)
\in V_{h} \times Q_{h}$ such that
\begin{equation}
  \label{p3:eq:stokes-olm}
  A_h(\bfu_h, p_h ; \bfv_h, q_h) = L_h(\bfv_h, q_h) \quad \foralls
  (\bfv_h,q_h) \in V_h \times Q_h.
\end{equation}
\begin{remark}
  The forms $s_h$, $S_h$ and the mesh $\meshast_1$ are defined such
  that contributions from both meshes are taken into account in the
  overlap region $\OmegaO$. These twice-counted contributions allows
  us to obtain stability and condition number estimates; this will be
  discussed further in Section~\ref{p3:sec:stability-props}
  and~\ref{p3:sec:condition-number}.
\end{remark}
We conclude this section by stating the following Galerkin
orthogonality relation.
\begin{proposition}
  \label{p3:prop:galerkin-ortho}
  Let $(\bfu, p) \in [H^2(\Omega)]^d \times H^1(\Omega)$
  be a (weak) solution of~\eqref{p3:eq:strongform}
  and $(\bfu_h, p_h) \in V_h \times Q_h$
  the solution of the finite element
  formulation~\eqref{p3:eq:stokes-olm}. Then
  \begin{equation}
    A_h(\bfu - \bfu_h, p - p_h; \bfv_h, q_h)
    = 0 \quad \foralls (\bfv_h, q_h) \in  V_h \times Q_h.
    \label{p3:eq:galerkin-ortho}
  \end{equation}
\end{proposition}
\begin{proof}
  The proof is standard and follows from integration by parts, the
  continuity conditions~\eqref{p3:eq:strong_cont_condition} --
  \eqref{p3:eq:strong_flux_condition} and observing that $s_h(\bfu,
  \bfv_h) \equiv 0 \equiv S_h(\bfu,p; \bfv_h,q_h)$ for all $(\bfv_h,
  q_h) \in V_h \times Q_h$.
  \qed
\end{proof}

\section{Approximation properties}
\label{p3:sec:approx-est}

In this section, we establish the appropriate interpolation operators
and provide interpolation estimates for use in
Sections~\ref{p3:sec:stability-props}--\ref{p3:sec:a-priori} for the
\apriori{} error analysis of the method proposed in
Section~\ref{p3:sec:stokes-olm-method}.

With reference to Section~\ref{p3:ssec:overlapping-meshes}, $\mcV_h$
will here denote some composite finite element space $\mcV_h =
\mcV_{h,1} \bigoplus \mcV_{h,2}$ comprising piecewise polynomial
functions or vector fields defined on $\meshast_i$, $i = 1, 2$. In
particular, we shall let $\mcV_h = V_h$ or $\mcV_h = Q_h$ in the
subsequent sections. The constants $C$ involved in the inequalities
will depend only on $\Omega_i$, the regularity of the function spaces
considered, and possibly the shape-regularity of $\mesh_i$ and the
polynomial order of $\mcV_{h,i}$, $i = 1, 2$.

\subsection{Norms}
\label{p3:ssec:norms}
Given the domain $\Omega = \OmcupOm$ and corresponding overlapping
meshes $\meshast_1$ and $\mesh_2$, we introduce the following pairs of
mesh-dependent norms for $(\bfv,q) \in [H^2(\Omega)]^d \times
H^1(\Omega)$ and $(\bfv_h, q_h) \in V_h \times Q_h$
(see~\eqref{p3:eq:elementspaces} and~\eqref{p3:eq:broken}):
\begin{align}
  \label{p3:eq:triple-v-norm}
  \tn \bfv \tn^2 &= \| \nabla \bfv \|^2_{\Omega_1 \cup \Omega_2} +
  \ifnorm{\meanvalue{\nablan\bfv}}{-1/2}^2 +
  \ifnorm{\jump{\bfv}}{1/2}^2,
  \\
  \label{p3:eq:triple-v-normast}
  \tn \bfv \tnast^2 &= \| \nabla \bfv \|^2_{\mcupm} +
  \ifnorm{\meanvalue{\nablan\bfv}}{-1/2}^2 +
  \ifnorm{\jump{\bfv}}{1/2}^2,
  \\
  \label{p3:eq:triple-p-norm}
  \tn q \tn^2 &= \| q \|_{{\OmcupOm}}^2,
  \\
  \tn q \tnast^2 &= \| q \|_{{\mcupm}}^2,
  \label{p3:eq:triple-p-normast}
\end{align}
where
\begin{align}
  \ifnormalpha{\bfv}^2 &= \sum_{T \in \mesh_2} h_{T}^{-2\alpha} \| \bfv
  \|^2_{0,\Gamma \cap T} ,
  \label{p3:eq:alpha-norm}
\end{align}
and finally
\begin{align}
  \label{p3:eq:triple-up-norm}
  \tn (\bfv,q) \tn^2 &= \tn \bfv \tn^2 + \tn q \tn^2,
  \\
  \tn (\bfv,q) \tn^2_{\ast} &= \tn \bfv \tn^2_{\ast} +
  \|q\|^2_{\ast} .
  \label{p3:eq:triple-up-normast}
\end{align}
Note that $\tn \cdot \tn_{\ast}$-norms are defined on the regular
meshes $\meshast_1$ and $\mesh_2$ and therefore represent proper norms
for the discrete finite element functions, while $\tn \cdot \tn$ is
more suitable for functions defined on $ \Omega = \OmcupOm$ only.

\subsection{Trace inequalities and inverse estimates}
\label{p3:ssec:trace-inverse}
In the following, $\meshast$ denotes always one of the regular,
non-intersected meshes $\meshast_1$ and $\meshast_2$.  We first recall
the following trace inequalities for $v \in H^1(\Omega)$ (or $v \in
[H^1(\Omega)]^d$)~\citep{HansboHansboLarson2003}:
\begin{alignat}{1}
  \| v \|_{\partial T} &\leqslant  C (h_{T}^{-1/2} \| v \|_{T}
  + h_{T}^{1/2} \| \nabla v \|_{T} ) \quad \foralls T \in \meshast,
  \label{p3:eq:trace-inequality}
  \\
  \| v \|_{T \cap \Gamma} &\leqslant  C (h_{T}^{-1/2} \| v \|_{T}
  + h_{T}^{1/2} \| \nabla v \|_{T} ) \quad \foralls T \in \meshast.
  \label{p3:eq:trace-inequality-for-FD}
\end{alignat}
The following inverse estimates for $v \in \mcV_h = \mcV_h(\meshast)$
will also be frequently
used. See~\citep{Quarteroni2009,HansboHansboLarson2003,HansboHermansson2003})
for proofs.
\begin{alignat}{3}
  \label{p3:eq:inverse-estimates-for-triangles}
  h_T \| \nabla v_h \|_T &\leqslant C
  \| v_h \|_{T} &\quad &\foralls T \in \meshast,
  \\
  \label{p3:eq:inverse-estimate-for-facets}
  h_F^{1/2} \| \nablan v_h \|_{F} &\leqslant C \| \nabla v_h
  \|_{T} &\quad &\foralls T \in \meshast, \\
  \label{p3:eq:inverse-estimate-for-fd-traces}
  h_F^{1/2} \| \nablan v_h \|_{\Gamma \cap T} &\leqslant C \|
  \nabla v_h \|_{T}  &\quad &\foralls T \in \meshast.
\end{alignat}
We emphasize that $h_F$
in~\eqref{p3:eq:inverse-estimate-for-fd-traces} (as
in~\eqref{p3:eq:inverse-estimate-for-facets}) denotes the diameter of
the \emph{entire} facet $F$ as part of the mesh $\meshast$.  Finally,
the estimation of the stabilization terms~\eqref{p3:eq:Sh-olm} will
involve the inverse estimates
\begin{alignat}{3}
  \sum_{T\in\meshast} h_T^2 \| \nabla v_h \|^2 &\leqslant C \| v_h \|_{\meshast}^2 &\quad
  &\foralls v_h \in \mcV_h,
  \label{p3:eq:inverse-estimate-grad-l2}
  \\
  \sum_{T\in\meshast} h_T^2 \| \Delta v_h \|^2 &\leqslant C \| \nabla v_h \|_{\meshast}^2 &\quad
  &\foralls v_h \in \mcV_h.
  \label{p3:eq:inverse-estimate-delta-grad}
\end{alignat}

To the end of this section, we state a Poincar\'e type inequality for
finite element functions on the overlapped mesh domain $\meshast_1$.
To prove the proposition, we will need the following lemma, which is
stated and proved in~\citep{Massing2012c}:
\begin{lemma}
  \label{p3:lem:l2norm-control-via-jumps}
  Let $\mesh = \{T\}$ be a mesh consisting of shape-regular elements
  $T$ and take any two elements $T_1$ and $T_2$ sharing a face
  $F$. Furthermore, let $u$ be a piecewise polynomial function defined
  on the macro-element $\mathcal{M} = T_1 \cup T_2$. There exists a
  constant $C$ depending only on the mesh quality parameters and the
  maximal polynomial order $p$ of $u$ restricted to each of the
  elements, such that
  \begin{equation}
    \| u \|_{T_1}^2 \leqslant C ( \|u\|_{T_2}^2
    + \sum_{j \leqslant p} h_F^{2 j + 1}
    (\jump{\nablan^{j} u},\jump{\nablan^{j} u})_F).
    \label{p3:eq:l2norm-control-via-jumps}
  \end{equation}
\end{lemma}
Also note that the inverse
inequalities~\eqref{p3:eq:inverse-estimates-for-triangles}
and~\eqref{p3:eq:inverse-estimate-for-facets} can be generalized, by
the standard scaling argument, to inverse estimates of the type
\begin{align}
  \label{p3:eq:general-inverse-estimate-for-elements}
  h_T^j \|D^j u \|_T &\leqslant C h_T^{j-i} \| D^{j-i} u \|_T
  \quad \text{for } i \leqslant j, \\
  \label{p3:eq:general-inverse-estimate-for-facets}
  h_F^{1/2} \| \nablan^{j} u \|_F &\leqslant C  \| D^j u \|_{T} ,
\end{align}
where $D^j u $ denotes the $j$-th total derivative of $u$.
\begin{proposition}
  \label{p3:prop:l2-norm-on-olm}
  Let $\Omega_1$ be the overlapped domain and $\meshast_1$ be the
  overlapped mesh, and let $\mcV_{h,1}$ be a finite element space on
  $\meshast_1$ consisting of continuous piecewise polynomials.
  Then
  \begin{align}
    \| v_h \|_{\meshast_1}^2
    &\leqslant C ( \| v_h \|_{\Omega_1}^2
    + \sum_{T \in \meshast_1} h_T^2 \| \nabla v_h \|_T^2) ,
    \label{p3:eq:poincare-for-meshast-1} \\
    \| v_h \|_{\meshast_1}^2
    &\leqslant C ( \| v_h \|_{\Omega_1}^2 + \| \nabla v_h \|_{\meshast_1}^2) ,
    \label{p3:eq:poincare-for-meshast-2}
  \end{align}
  for all $v_h \in \mcV_{h, 1}$.
\end{proposition}
\begin{proof}
  By definition
  \begin{equation}
    \label{p3:eq:l2-norm-on-olm:step1}
    \| v_h \|_{\meshast_1}^2
    = \sum_{T \in \mesh_{0, 1}} \| v_h \|_{T}^2
    + \sum_{T_1 \in \mesh_{0, \Gamma}} \| v_h \|_{T_1}^2 .
  \end{equation}
  Applying Lemma~\ref{p3:lem:l2norm-control-via-jumps} and
  subsequently
  invoking~\eqref{p3:eq:general-inverse-estimate-for-facets}
  and~\eqref{p3:eq:general-inverse-estimate-for-elements} give
  \begin{equation}
    \label{p3:eq:l2-norm-on-olm:step2}
    \sum_{T_1 \in \mesh_{0, \Gamma}} \| v_h \|_{T_1}^2
    \leqslant C \sum_{T \in \meshast_1}
    \left ( \| v_h \|^2_{T}
    + h_T^2 \| \nabla v_h \|_T^2  \right ) ,
  \end{equation}
  where the shape-regularity of $\meshast_1$ provides a bound on $h_F$
  in terms of $h_T$. Combining~\eqref{p3:eq:l2-norm-on-olm:step1}
  and~\eqref{p3:eq:l2-norm-on-olm:step2}
  yields~\eqref{p3:eq:poincare-for-meshast-1}.

  Since $h_T \leqslant |\Oast_1 | $, the second Poincar\'e type
  inequality~\eqref{p3:eq:poincare-for-meshast-2} is a simple
  consequence of the first one.
  \qed
\end{proof}

\subsection{Interpolation estimates}
\label{p3:ssec:interpolation-est}

The aim of this section is to construct and analyze an interpolation
operator $\pi_h : L^2(\Omega) \rightarrow \mcV_h$ and a corrected
interpolation operator $\pi_h^c: [H^1(\Omega)]^d \rightarrow V_h$. The
construction is based on extending the
standard Scott--Zhang interpolation operator~\citep{ScottZhang1990}.

First, for $s \geqslant 0$, there exists a linear extension operator
$\mcE^s:H^s(\Omega_1) \rightarrow H^s(\Oast_1)$ and a constant $C > 0$
such that $\mcE^s v |_{\Omega_1} = v$ and
\begin{equation}\label{p3:stabext}
  \| \mcE^s v \|_{s,\Oast_1} \leqslant C \| v \|_{s,\Omega_1}
\end{equation}
for all $v \in H^s(\Omega_1)$~\citep{Stein1970}.

Next, for $i = 1, 2$, let $\pi_{h,i} : L^2(\Oast_i) \rightarrow
\mcV_{h,i}$ be the Scott--Zhang interpolation operator and recall the
standard interpolation error estimates~\citep{ScottZhang1990}:
\begin{alignat}{3}
  \| v_i - \pi_{h,i} v_i \|_{r,T} &\leqslant C h^{s-r}| v_i |_{s,\omega(T)},
  &\quad 0\leqslant r \leqslant s \quad &\foralls T\in \meshast_i,
  \label{p3:eq:interpest0}
  \\
  \| v_i - \pi_{h,i} v_i \|_{r,F} &\leqslant C h^{s-r-1/2}| v_i
  |_{s,\omega(F)},
  &\quad 0\leqslant r \leqslant s \quad & \foralls F \in \partial_i \meshast_i,
  \label{p3:eq:interpest1}
\end{alignat}
for all $v_i \in H^s(\Omega_i^*)$ for $i = 1, 2$. Here, $\omega(T)$ is
the patch of all elements sharing a vertex with $T$; the patch
$\omega(F)$ of a face $F$ is defined analogously.

Altogether, we define
\begin{equation}
  \label{p3:eq:def:interp}
  \pi_{h}: H^s(\Omega_1) \bigoplus
  H^s(\Omega_2) \rightarrow \mcV_{h}, \quad s \geqslant 1,
\end{equation}
by
\begin{equation}
  \pi_{h} v = \pi_{h,1} \mcE^s v_1 \oplus \pi_{h,2} v_2 .
\end{equation}
The estimate~\eqref{p3:stabext} together with the interpolation error
estimate~\eqref{p3:eq:interpest0} for the Scott--Zhang interpolation
operator imply the following interpolation estimates:
\begin{alignat}{3}
  \| v - \pi_h v \|_{r,T} &\leqslant C h^{s-r}| v |_{s,\omega(T)},
  &\quad 0\leqslant r \leqslant s \quad &\foralls T \in \mesh_1 \cup \mesh_2,
  \label{p3:eq:interpest0-cutmes}
  \\
  \label{p3:eq:interpest1-cutmes}
  \| v - \pi_h v \|_{r,F} &\leqslant C h^{s-r-1/2}| v |_{s,\omega(T)},
  &\quad 0\leqslant r \leqslant s \quad &\foralls F \in \partial_i \mesh_1
  \cup \partial_i \mesh_2.
\end{alignat}

We now return to our specific finite elements spaces $V_h$ and
$Q_h$. We continue writing $\pi_h$ for both the interpolation
operator~\eqref{p3:eq:def:interp} defined for $[H^{k+1}(\Omega)]^d
\rightarrow V_h^{k}$ and $H^{l+1}(\Omega) \rightarrow Q_h^l$.  The
following lemma provides an interpolation error estimate in the $\tn
\cdot \tn$ norms.
\begin{lemma}
  There is a constant $C > 0$ such that for all $\bfv \in
  [H^{k+1}(\Omega)]^d$, $q \in H^{l+1}(\Omega)$ it holds that
  \begin{align}
    \tn \bfv - \pi_h \bfv \tn &\leqslant C h^{k}  | \bfv
    |_{k+1,\Omega},
    \label{p3:eq:interpest-v}
    \\
    \tn q - \pi_h q\tn &\leqslant C h^{l+1} | q |_{l+1,\Omega},
    \label{p3:eq:interpest-p}
    \\
    \tn (\bfv - \pi_h \bfv, q - \pi_h q) \tn
    &\leqslant C ( h^{k} | \bfv |_{k+1,\Omega}
    + h^{l+1} | q |_{l+1,\Omega} ).
    \label{p3:eq:interpest-vp}
  \end{align}
\end{lemma}
\begin{proof}
  For the proof of~\eqref{p3:eq:interpest-v} we refer
  to~\citet{HansboHansboLarson2003}. The
  estimate~\eqref{p3:eq:interpest-p} follows from the definition of
  $\pi_h$, the interpolation estimate~\eqref{p3:eq:interpest0} and the
  continuity of the extension operator $\mcE^{s}$.
  \qed
\end{proof}

We conclude the section by referring to a modified Scott--Zhang
interpolation operator introduced by~\citet{BeckerBurmanHansbo2009} to
establish a Nitsche extended finite element method for incompressible
elasticity. There, the standard Scott-Zhang interpolant was corrected
to gain some additional orthogonality properties on the
interface. Although only used for continuous piecewise linear velocity
fields in combination with piecewise constant pressures
in~\citet{BeckerBurmanHansbo2009}, we will show in the next section
that the interpolant can be utilized to prove a sufficient inequality
for $b_h(\cdot,\cdot)$ adapted to the case of overlapping meshes.

To construct the corrected interpolant, \citet{BeckerBurmanHansbo2009}
regrouped the elements $T \in \mesh_{0,\Gamma}$ into a collection
$\Gamma_{\omega}$ of patches ${\omega_i}$ in such a way that each
patch can be associated with a hat function $\Phi_i$ constructed from
the finite element functions defined on elements $T \subset \omega_i$.
Moreover, $\omega_i$ and $\Phi_i$ have the following three properties:
there are constants $c, C > 0$ such that
\begin{enumerate}
  \item[(i)]
    $c h \leqslant \diam(\omega_i) \leqslant C h $
  \item[(ii)] $c h \leqslant \int_{\Gamma \cap \omega_i} \Phi_i \leqslant
    C h$
  \item[(iii)]$ c h^{-1} \leqslant | \nabla \Phi_i(x) | \leqslant C h^{-1}$
\end{enumerate}
Loosely speaking, these properties guarantee that patches and hat
functions behave as standard elements and shape functions.  Although
\citet{BeckerBurmanHansbo2009} only elaborate on the construction in two
dimension and for a non-fitting interface, a similar construction is
possible in any spatial dimension $d$ and for the case where the
interface is defined by the boundary of another mesh.

By adjusting the degrees of freedom associated with the hat functions
$\{\Phi_i\}$, an interpolation operator
\begin{equation*}
\pi^c_{h,1}:[H^1(\Oast_1)]^d \rightarrow V_{h,1}
\end{equation*}
can be defined that satisfies an orthogonality property of the form
\begin{equation}
  \label{p3:eq:scott-zhang-interpolant-orthogonality}
  (\bfv_1-\pi^c_{h,1} \bfv_1, \bfn)_{\Gamma \cap \omega_i} = 0
  \quad \foralls \omega_i ,
\end{equation}
and an interpolation estimate of the form
\begin{equation}
  \| \bfv_1 - \pi^c_{h,1} \bfv_1 \|_{0, \Oast_1}
  + h \|\nabla(\bfv_1-\pi^c_{h,1} \bfv_1) \|_{0,\Oast_1}
  \leqslant C h \|\nabla \bfv \|_{0,\Oast_1} .
  \label{p3:eq:scott-zhang-interpolant-estimate}
\end{equation}
Finally, we can construct a corrected Scott--Zhang interpolation
operator $\pi_h^c$ for the composite velocity spaces $V_h$ by
\begin{equation}
  \begin{split}
    \label{p3:eq:def:corrected:interpolant}
    \pi_h^c: [H^1(\Omega)]^d &\rightarrow V_{h,1} \bigoplus V_{h,2} \\
    \pi_h^c(\bfv) &= \pi^c_{h,1} \mcE^1 \bfv_1 \oplus \pi_{h,2} \bfv_2
  \end{split}
\end{equation}
where $\pi_{h,2}$ is the standard Scott--Zhang operator on $\Oast_2$.

We will need the following continuity property of the corrected
interpolation operator $\pi^c_h$ with respect to different norms.
\begin{lemma}
  \label{p3:lem:interpolant-continuity}
  Let $\bfv \in [H^1(\Omega)]^d$. There exists a constant $C > 0$ such
  that
  \begin{align}
    \tn \pi^c_h \bfv \tn_{\ast} \leqslant C \| \bfv \|_{1,\OmcupOm}.
    \label{p3:eq:interpolant-continuity}
  \end{align}
\end{lemma}
\begin{proof}
  By definition,
  \begin{equation*}
  \tn \pi^c_h \bfv_h \tnast^2 = \| \nabla \pi^c_h \bfv_h \|_{\mcupm}^2
  + \ifnorm{ \langle \nablan \pi^c_h \bfv_h \rangle }{-1/2}^2
  + \ifnorm{\jump{\pi^c_h \bfv_h}}{1/2}^2.
  \end{equation*}
  The estimate for the first term follows directly from the definition
  (and boundedness) of $\pi^c_h$ and the continuity of the extension
  operator $\mcE$. Using the inverse
  estimate~\eqref{p3:eq:inverse-estimate-for-fd-traces}, the bound for
  the second follows from the bound of the first.  To estimate the
  last term, we first note that $ \ifnorm{\jump{\pi^c_h \bfv}}{1/2} =
  \ifnorm{\jump{\pi^c_h \bfv-\bfv}}{1/2} $ since $ \jump{\bfv} = 0 $
  for $ \bfv \in H^1(\Omega) $ and thus
  \begin{align*}
    \ifnorm{\jump{\pi^c_h \bfv}}{1/2}^2
    \leqslant
    \ifnorm{ \pi^c_{h} \bfv - \bfv_1}{1/2}^2
    + \ifnorm{ \pi^c_{h} \bfv - \bfv_2}{1/2}^2.
  \end{align*}
  Another application of the trace inequality
  \eqref{p3:eq:trace-inequality-for-FD} combined with interpolation
  estimates~\eqref{p3:eq:scott-zhang-interpolant-estimate} yields for
  $i = 1, 2$:
  \begin{align*}
    \ifnorm{\pi^c_{h} \bfv - \bfv_i}{-1/2}^2
    &\leqslant
    \sum_{T \in \meshast_i}
    h^{-1}_{T} \| \pi^c_h \bfv - \bfv \|_{\Gamma \cap T}^2
    \\
    &\leqslant
    \sum_{T \in \meshast_i}
    h^{-1} ( h^{-1} \| \pi^c_h \bfv - \bfv \|_{ T}^2
    + h \| \nabla ( \pi^c_h \bfv - \bfv) \|_{ T}^2)
    \\
    &\leqslant \| \bfv \|_{1,\meshast_i}^2
    \leqslant C \| \bfv \|_{1, \OmcupOm}^2.
  \end{align*}
  \qed
\end{proof}

\section{Stability estimates}
\label{p3:sec:stability-props}

In this section, we first prove that the form $a_h + s_h$ is
continuous and coercive with respect to the $\tn \cdot \tnast$ norm in
Proposition~\ref{p3:prop:a_h-stabprop}; that $b_h$ is continuous with
respect to both the $\tn \cdot \tn$ and the $\tn \cdot \tnast$ norms
in Proposition~\ref{p3:prop:b_h-stabprop}; and an additional
inequality for $b_h$ in
Proposition~\ref{p3:prop:stabest-b}. Subsequently, these estimates are
used to show that the complete stabilized Nitsche form $A_h$ satisfies
the Babu\v ska--Brezzi inf-sup condition in
Theorem~\ref{p3:thm:stability}.

We begin by demonstrating that $a_h + s_h$ is continuous and stable
with respect to the norm $\tn \cdot \tnast$.
\begin{proposition}
  \label{p3:prop:a_h-stabprop}
  Let $a_h$ and $s_h$ be defined by~\eqref{p3:eq:ah-olm}
  and~\eqref{p3:eq:sh-olm}. There are constants $c > 0$ and $C > 0$
  such that
  \begin{alignat}{5}
    a_h(\bfu_h,\bfv_h) + s_h(\bfu_h,\bfv_h) &\leqslant C \tn \bfu_h \tn_{\ast} \;
    \tn \bfv_h \tn_{\ast} &\quad &\foralls
    \bfu_h,\bfv_h \in V_h,
    \label{p3:eq:a_h-stable-2}
    \\
    c \tn \bfv_h \tn_{\ast} ^2 &\leqslant a_h(\bfv_h, \bfv_h) + s_h(\bfv_h,
    \bfv_h) &\quad &\foralls \bfv_h \in V_h.
    \label{p3:eq:a_h-stable-3}
  \end{alignat}
\end{proposition}
\begin{proof}
  A proof in the absence of the term $s_h$ and with $\tn \cdot \tnast$
  replaced by $\tn \cdot \tn$ is given in
  \citet{HansboHansboLarson2003}.  Since by definition $\tn \bfv_h
  \tnast^2 = \tn \bfv_h \tn^2 + \| \nabla \bfv_h \|_{\OmegaO}^2$ the
  estimates~\eqref{p3:eq:a_h-stable-2} and~\eqref{p3:eq:a_h-stable-3}
  follow directly from the equivalence $\| \nabla \bfv_h \|^2_{\mcupm}
  \sim \| \nabla \bfv_h \|^2_{\OmcupOm} + \|\nabla(\bfv_{h,1} -
  \bfv_{h,2}) \|_{\OmegaO}^2$ and the definition of $s_h$.
  \qed
\end{proof}

Next, we state the continuity properties of $b_h$ with respect to the
various norms.
\begin{proposition}
  \label{p3:prop:b_h-stabprop}
  Let $b_h$ be defined by~\eqref{p3:eq:bh-olm}. There exists a
  constant $C > 0$ such that
  \begin{alignat}{2}
    b_h(\bfv, q) &\leqslant C \tn \bfv \tn \, \tn q \tn
    &\quad &\foralls (\bfv,q) \in [H^1(\Omega)]^3 \times L^2(\Omega),
    \label{p3:eq:b-cont-wrt-norm-1}
    \\
    b_h(\bfv_h, q_h) &\leqslant C \tn \bfv_h \tn_{\ast} \tn q_h \tn_{\ast}
    &\quad &\foralls (\bfv_h,q_h) \in V_h \times Q_h.
    \label{p3:eq:b-cont-wrt-norm-2}
  \end{alignat}
\end{proposition}
\begin{proof}
  The first inequality is a trivial consequence of the fact that the
  interface term $\jump{\bfv}$ vanishes for $\bfv \in
  [H^1(\Omega)]^3$.  To prove the second estimate, it is enough to
  estimate the interface term $(\bfn \cdot \jump{\bfv_h},
  \meanvalue{q_h})_{\Gamma}$. Recall that we chose $\meanvalue{q_h} =
  q_{h,2}$.  We combine the Cauchy--Schwarz inequality
  with~\eqref{p3:eq:trace-inequality-for-FD},
  \eqref{p3:eq:inverse-estimates-for-triangles}
  and~\eqref{p3:eq:inverse-estimate-for-fd-traces} to obtain
  \begin{align*}
    (\meanvalue{q_h},\bfn\cdot\jump{\bfv_{h}})_{\Gamma}
    &\leqslant
    |(q_{h,2},\bfn \cdot \bfv_{h,1})_{\Gamma} |
    + |(q_{h,2}, \bfn \cdot \bfv_{h,2})_{\Gamma} | \\
    &\leqslant
    \ifnorm{q_{h,2}}{-1/2}
    \left (
    \ifnorm{\bfv_{h,1}}{1/2}
    +
    \ifnorm{\bfv_{h,2}}{1/2}
    \right )
    \\
    &\leqslant C
    \| q_{h,2} \|_{\mesh_2}
    \left (
    \| \nabla \bfv_{h,1} \|_{\meshast_1}
    +
    \| \nabla \bfv_{h,2} \|_{\mesh_2} \right ) \\
    &\leqslant C( \tn q_h \tnast
    \tn \bfv_h \tnast).
  \end{align*}
  \qed
\end{proof}

The proof of the inf-sup condition for stabilized Stokes elements
often involves an inequality known as Verf\"urths trick
\citep{Verfurth1994a}.
In the next proposition, we present a version adapted to the case
of overlapping meshes.
\begin{proposition}
  \label{p3:prop:stabest-b}
  There are constants $C_1 > 0$ and $C_2 > 0$ such that for all $q_h
  \in Q_h$:
  \begin{equation}
    \sup_{\bfv_h \in V_h} \frac{b(\bfv_h, q_h)}{\tn \bfv_h \tnast }
    \geqslant C_1 \| q_h \|_{\OmcupOm}
    - C_2 (\sum_{T \in \mcupm} h_T^2 \| \nabla q_h \|^2_{T})^{1/2}.
    \label{p3:eq:bad-inequality}
  \end{equation}
\end{proposition}
\begin{proof}
  Let $q_h \in Q_h$ be given. The $\Div$-operator maps
  $[H^1(\Omega)]^d$ onto $ L^2(\Omega)$ and there is a constant $C >
  0$ such that $\| \bfv \|_{1,\Omega} \leqslant C \| \Div \bfv
  \|_{\Omega}$~\citep{GiraultRiviereWheeler2005} for all
  $\bfv$. Choosing $\bfv \in [H^1(\Omega)]^d$ such that $\Div \bfv = -
  q_h$ thus gives
  \begin{equation}
  b_h(\bfv,q_h) =
  \|q_h\|^2_{\Omega} \geqslant C^{-1}\| q_h \|_{\Omega} \| \bfv \|_{1,\Omega},
  \end{equation}
  where we have also used the fact that $\jump{\bfv} = 0$.

  Next, we take $\bfv_h = \picorr_h \bfv$, where
  $\picorr_h$ is the corrected Scott--Zhang interpolant introduced in
  Section~\ref{p3:ssec:interpolation-est}.  It follows that
  \begin{align}
    b_h(\bfv_h, q_h) &\equiv
    b_h(\picorr_h \bfv, q_h) =
    b_h(\picorr_h \bfv - \bfv, q_h)
    + b(\bfv, q_h)
    \\
    &\geqslant
    b_h(\picorr_h \bfv - \bfv, q_h)
    + C^{-1} \|q_h \|_{\Omega} \| \bfv \|_{1,\Omega}.
    \label{p3:eq:bh-with-tildev}
  \end{align}
  To proceed further, we recall that $\meanvalue{q_h} \equiv
  q_{h,2}$, $\bfn \equiv \bfn_2$ and observe that
  \begin{align*}
    \sum_{i=1}^2 (\bfn_i \cdot (\picorr_{h, i} \bfv_i -
    \bfv_i),q_{h,i})_{\Gamma}
    =
    (\bfn \cdot \jump{\picorr_{h} \bfv -
    \bfv},q_{h,2})_{\Gamma} +
    (\bfn \cdot (\picorr_{h, 1} \bfv_1 - \bfv_1),
    \jump{q_h})_{\Gamma}.
  \end{align*}
  Using this identity and integrating the first term
  in~\eqref{p3:eq:bh-with-tildev} by parts, we derive that
  \begin{align*}
    b_h(\picorr_h \bfv - \bfv, q_h)
    &\equiv
    - (\nabla \cdot (\picorr_{h} \bfv - \bfv), q_h)_{\OmcupOm}
    + (\bfn \cdot
    \jump{ \picorr_h \bfv - \bfv}, q_{h,2})_{\Gamma}
    \\
    &=
    \underbrace{
      \sum_{i = 1}^2 \sum_{T \in \mesh_i}
      (\picorr_h \bfv - \bfv,\nabla q_h)_{T \cap \Omega_i}
    }_{I}
    -
    \underbrace{(\bfn \cdot (\picorr_{h, 1} \bfv_1 -
    \bfv_1), \jump{q_{h}})_{\Gamma}}_{II}.
  \end{align*}
  Rewriting $(\picorr_h \bfv - \bfv,\nabla q_h)_{T\cap \Omega} =
  (h_T^{-1} ( \picorr_h \bfv - \bfv),h_T \nabla q_h)_{T \cap \Omega}$
  and applying the interpolation
  estimates~\eqref{p3:eq:interpest0-cutmes}
  and~\eqref{p3:eq:scott-zhang-interpolant-estimate}, the first term
  $I$ can be bounded from below by
  \[
    |I| \geqslant - \| \nabla \bfv \|_{\mcupm} \bigl(\sum_{T \in \mcupm} h_T^2 \| \nabla q_h
    \|^2\bigr)^{1/2}.
  \]
  To estimate the second term $II$, we exploit the orthogonality
  property~\eqref{p3:eq:scott-zhang-interpolant-orthogonality} and
  write
  \begin{equation}
    \begin{split}
      (\bfn \cdot (\picorr_{h,1} \bfv_1 - \bfv_1), \jump{q_{h}})_{\Gamma}
      & =
      \sum_{\omega \in \Gamma_{\omega}}
      (\bfn \cdot (\picorr_{h, 1} \bfv_1 - \bfv_1), q_{h,2}
      - \bar{q}_{h,2})_{\Gamma \cap \omega}
      \\
      & \quad
      - \sum_{\omega \in \Gamma_{\omega}}
      (\bfn \cdot (\picorr_{h, 1} \bfv_1 - \bfv_1), q_{h,1}
      - \bar{q}_{h,1})_{\Gamma \cap \omega} ,
      \label{p3:eq:sum-with-meanvalue}
    \end{split}
  \end{equation}
  where $\bar{q}_{h,i} |_{\omega} = \tfrac{1}{|\omega|}\int_{\omega}
  q_{h,i} \dx$, $i=1,2$. Now each contribution in the
  sum~\eqref{p3:eq:sum-with-meanvalue} can bounded by
  \begin{equation}
    \begin{split}
    \label{p3:eq:sum-contrib}
    (\bfn \cdot (\picorr_{h,1} \bfv_1 - \bfv_1)&,
    q_{h,i} - \bar{q}_{h,i})_{\Gamma \cap \omega} \\
    & \leqslant
    \| \bfn \cdot (\picorr_{h,1} \bfv_1 - \bfv_1)\|_{1/2,h,\Gamma
    \cap \omega}
    \| q_{h,i} - \bar{q}_{h,i} \|_{-1/2,h,\Gamma \cap \omega}
    \end{split}
  \end{equation}
  Using again the trace
  inequality~\eqref{p3:eq:trace-inequality-for-FD}, the last factor
  can estimated by
  \begin{equation*}
    \| q_{h,i} - \bar{q}_{h,i} \|_{-1/2,h,\Gamma \cap \omega}
    \leqslant C | \omega | \|\nabla q_{h,i} \|_{\omega}
    \leqslant C h_T        \|\nabla q_{h,i} \|_{\omega},
  \end{equation*}
  since $q_{h, i} - \bar{q}_{h, i}$ has mean value zero on $\omega$
  and so $\| q_{h, i} - \bar{q}_{h, i} \|_{\omega} \leqslant |\omega| \| \nabla
  q_{h, i} \|_{\omega}$. To arrive at an estimate for first factor
  in~\eqref{p3:eq:sum-contrib}, we use the trace
  inequality~\eqref{p3:eq:trace-inequality-for-FD} in combination with
  the interpolation
  estimate~\eqref{p3:eq:scott-zhang-interpolant-estimate} and the
  continuity of the extension operator:
  \begin{equation*}
    \| \bfn \cdot (\picorr_{h,1} \bfv_1 - \bfv_1)
    \|_{1/2,h,\Gamma\cap\omega}
    \leqslant
    \| \nabla \bfv_1 \|_{\omega\cap\Omega_1}.
  \end{equation*}
  Thus
  \begin{align*}
    | (\bfn \cdot ( \picorr_{h,1} \bfv_1 - \bfv_1), \jump{q_{h}})_{\Gamma} |
    &\leqslant \sum_{\omega \in \Gamma_{\omega}}
    h_T(\| \nabla q_{h,1} \|_{\omega}
    +\| \nabla q_{h,2} \|_{\omega})
    \| \nabla \bfv_1 \|_{\omega\cap\Omega_1}
    \\
    &\leqslant C (\sum_{T\in\mcupm} h_T^2 \| \nabla q_h \|_{T}^2)^{1/2}
    \| \nabla \bfv \|_{\Omega_1}.
  \end{align*}
  Putting $I$ and $II$ together, we see that
  \begin{equation*}
    | b_h(\picorr_h \bfv - \bfv, q_h) |
    \leqslant
    C (\sum_{T\in\mcupm} h_T^2 \| \nabla q_h \|_{T}^2)^{1/2}
    \| \nabla \bfv \|_{\Omega},
  \end{equation*}
  which together with \eqref{p3:eq:bh-with-tildev}
  gives the estimate
  \[
    \frac{b(\picorr_h \bfv, q_h)}{\| \bfv \|_{\Omega} }
    \geqslant \widetilde{C}_1 \| q_h \|_{\OmcupOm}
    - \widetilde{C}_2 (\sum_{T \in \mcupm} h_T^2 \| \nabla q_h \|^2_{T})^{1/2}
  \]
  Now the estimate follows from the
  continuity~\eqref{p3:eq:interpolant-continuity} of the operator
  $\picorr_h$.
  \qed
\end{proof}
\begin{remark}
  Although looking very similar to the standard ``bad-inequality'',
  the inequality~\eqref{p3:eq:bad-inequality} bears some subtle
  differences. Note that we count contributions $\|\nabla q_h \|_T$
  twice in the overlap domain $\Omega_O$.  This is mainly a
  consequence of the trace
  inequality~\eqref{p3:eq:trace-inequality-for-FD}.  The $L^2$-norm of
  $q_h$ on the other hand is only taken on $\OmcupOm$; thus requiring
  special consideration in the subsequent proof of the inf-sup
  estimate for $A_h$ using the norm $\tn q_h \tnast = \| q_h
  \|_{\mcupm}$.
\end{remark}

We conclude the section by proving the inf-sup condition for the
bilinear form $A_h$ with respect to the norm $\tn (\cdot,\cdot) \tnast$.
\begin{theorem}
  Let $A_h$ be defined by~\eqref{p3:eq:Ah-olm}--\eqref{p3:eq:Sh-olm}
  and assume that either $\{\alpha,\beta\} = \{1,1\}$ or
  $\{\alpha,\beta\} = \{-1,1\}$.
  Then there is a constant $c$ such that
  \begin{equation}
    \sup_{(\bfv_h, q_h) \in V_h \times Q_h } \frac{A_h(\bfu_h, p_h;
    \bfv_h,q_h )}{\tn (\bfv_h,q_h) \tn_{\ast}} \geqslant  c \tn
    (\bfu_h,p_h) \tnast
    \quad \foralls (\bfu_h,p_h) \in V_h \times Q_h.
    \label{p3:eq:stability}
  \end{equation}
  \label{p3:thm:stability}
\end{theorem}
\begin{proof}
  The proof follows the presentation given in~\citet{Franca1993}.  Let
  $(\bfu_h, p_h)$ be given and consider the case $\{\alpha, \beta \} =
  \{1,1\}$.

  \emph{Step 1}: Choose $\bfv_h = -\bfw$ such that the supremum in
  \eqref{p3:eq:bad-inequality} is attained. By scaling $\bfw$, we
  can assume that $\tn \bfw \tnast = \| p_h \|_{\OmcupOm}$.
  Inserting $(\bfv_h,q_h) = (-\bfw,0)$ into $A_h$, we obtain
  via~\eqref{p3:eq:a_h-stable-2} and~\eqref{p3:eq:bad-inequality}:
  \begin{align}
    A_h(\bfu_h,p_h; \bfv_h, q_h)
    &= - a_h(\bfu_h,\bfw) - s_h(\bfu_h, \bfw)
    - b_h(\bfw,p_h) + S_h(\bfu_h,p_h;\bfw,0)
    \nonumber
    \\
    &\geqslant
    -C_a \tn \bfu_h \tnast \| p_h \|_{\OmcupOm}
    \nonumber
    \\
    &\quad
    +\{
      C_1 \| p_h \|_{\OmcupOm} - C_2 (\sum_{T \in \mcupm} h_T^2
      \| \nabla q_h \|_T^2)^{1/2}
    \} \| p_h \|_{\OmcupOm}
    \nonumber
    \\
    &\quad
    + \underbrace{\delta \sum_{T \in \mcupm} h^2_T (-\Delta \bfu_h + \nabla p_h,
      -\Delta \bfw)_T}_{I}.
    \label{p3:eq:Ah_estimate_last_line}
\end{align}
We estimate $I$ using Cauchy--Schwarz and the inverse
inequality~\eqref{p3:eq:inverse-estimate-delta-grad}:
\begin{align*}
  I
  &\geqslant
  - C_I \delta
  \| \nabla \bfw \|_{\mcupm}
  \Bigl(
  \| \nabla \bfu_h \|_{\mcupm}
  +
  \bigl(
  \sum_{T \in \mcupm} h_T^2
  \| \nabla p_h \|_T \bigr)^{1/2}
  \Bigr) \\
  &\geqslant
  - C_I \delta
  \| p_h \|_{\OmcupOm}
  \Bigl(
  \tn \bfu_h \tnast
  +
  \bigl( \sum_{T \in \mcupm} h_T^2
  \| \nabla p_h \|_T \bigr)^{1/2}
  \Bigr)
\end{align*}
where the last inequality above follows as $\| \nabla \bfv \|_{\mcupm}
\leqslant \tn \bfv \tnast$ for all $\bfv$ and $\tn \bfw \tnast = \| p_h
\|_{\OmcupOm}$. Inserting this final estimate for $I$
into~\eqref{p3:eq:Ah_estimate_last_line} and using the inequality $ab
\leqslant \epsilon a^2 + (4 \epsilon)^{-1} b^2$ for any $\epsilon > 0$, we
arrive at
\begin{align*}
  A_h(\bfu_h,p_h; - \bfw, 0)
  &\geqslant
  \, (C_1 - \epsilon(C_a + C_2 + \delta C_I))
  \| p_h \|_{\OmcupOm}^2
  \\
  &\quad
  - (C_a + \delta C_I) (4\epsilon)^{-1}
  \tn \bfu_h \tnast^2 \\
  &\quad- (C_2 + \delta C_I) (4\epsilon)^{-1} \sum_{T \in \mcupm} h_T^2
  \| \nabla p_h \|_T^2
  \\
  &\geqslant
  \widetilde{C}_1\| p_h \|_{\OmcupOm}^2
  - \widetilde{C}_2
  \tn \bfu_h \tnast^2
  - \widetilde{C}_3 \sum_{T \in \mcupm} h_T^2
  \| \nabla p_h \|_T^2
\end{align*}
with positive constants $\widetilde{C}_1, \widetilde{C}_2,
\widetilde{C}_3$ for an appropriate choice of $\epsilon$.

\emph{Step 2}: Now choose $(\bfv_h,q_h) =
(\bfu_h,-p_h)$. Proposition~\ref{p3:prop:a_h-stabprop} and
subsequently~\eqref{p3:eq:inverse-estimate-delta-grad} give that
\begin{align*}
  A_h(\bfu_h,p_h; \bfu_h, -p_h)
  &=
  a_h(\bfu_h, \bfu_h) + s_h(\bfu_h, \bfu_h) - S_h(\bfu_h, p_h; \bfu_h, - p_h)
  \\
  \phantom{=}\;
  &\geqslant
  c_a \tn \bfu_h \tnast^2
  - \delta \sum_{T \in \mcupm}
  h_T^2 (-\Delta \bfu_h + \nabla p_h, - \Delta \bfu_h - \nabla
  p_h)_T
  \\
  \phantom{=}\;
  &\geqslant
  (c_a  - \delta C_I) \tn \bfu_h \tnast^2
  + \delta \sum_{T \in \mcupm}
  h_T^2 \| \nabla p_h \|_T^2
  \\
  \phantom{=}\;
  &\geqslant
  \widetilde{C}\tn \bfu_h \tnast^2
  + \delta \sum_{T \in \mcupm}
  h_T^2 \| \nabla p_h \|_T^2
\end{align*}
with positive $\widetilde{C}$ as long as $0 < \delta < c_a C_I^{-1}$.

\emph{Step 3}: To complete the proof, we combine step 1 and step 2 by
taking $(\bfv_h, q_h) = (\bfu_h - \eta \bfw, -p_h)$ for some $\eta >
0$. By choosing $\eta$ sufficiently small,
\begin{align*}
  A(\bfu_h, p_h; \bfv_h, q_h)
  &\geqslant
  C_1\tn \bfu_h \tnast^2
  +C_2 \| p_h \|_{\OmcupOm}^2
  + C_3 \sum_{T \in \mcupm}
  h_T^2 \| \nabla p_h \|_T^2
  \\
  &\geqslant
  C_1\tn \bfu_h \tnast^2 +
  \widetilde{C}_2 \tn p_h \tnast^2
  \geqslant \min(C_1,\widetilde{C}_2) \tn (\bfu_h, p_h) \tnast^2,
\end{align*}
for some constants $C_1, C_2, C_3 > 0$ where we used
Proposition~\ref{p3:prop:l2-norm-on-olm} to find that for some
constant $\widetilde{C}_2 > 0$
\begin{equation*}
  C_2 \| p_h \|_{\OmcupOm}^2 + C_3 \sum_{T \in \mcupm} h_T^2\| \nabla p_h \|_T^2
  \geqslant \widetilde{C}_2 \tn p_h \tnast^2.
\end{equation*}
Since $\tn (\bfv_h, q_h) \tnast \leqslant \tn (\bfu_h, p_h) \tnast + \eta
\tn \bfw \tnast \leqslant (1 + \eta) \tn (\bfu_h, p_h) \tnast$, the estimate
\eqref{p3:eq:stability} follows.

The proof for $\{\alpha,\beta\} = \{-1,1\}$ differs only in the
derivation of the final estimate in step 2 and can be taken from
\citet{Franca1993} with the presented adaption to the overlapping
meshes formulation.
\qed
\end{proof}
\begin{remark}
  We would like to comment on the role of the different
  ghost-penalties. As remarked earlier, the
  bad-inequality~\ref{p3:eq:bad-inequality} ``mixes domains'' in the
  sense that it the contributions $\| \nabla q_h \|_T $ are taken from
  the elements in $\mcupm$, so that the contributions from the overlap
  region $\OmegaO$ are counted twice, while $\| q_h \|$ is evaluated
  only once on $\OmcupOm$. The role of the pressure terms in
  the least-squares ghost-penalty $\sum_{T \in \mcupm} h_T^2(-\Delta
  \bfu_h + \nabla p_h, -\Delta \bfv_h + \nabla q_h)_T$ is two-fold.
  First, they compensate the negative contributions $-\sum_{T \in
  \meshast_1} h_T^2 \| \nabla q_h \|_T^2$ in the bad-inequality.
  Secondly, they allow in combination with
  Lemma~\ref{p3:prop:l2-norm-on-olm} to pass from $\| q_h
  \|_{\OmcupOm} $ to $\| q_h \|_{\mcupm}$. We further note that the
  velocity terms in the least-squares ghost-penalty make it necessary
  (via the inverse
  inequality~\eqref{p3:eq:inverse-estimate-delta-grad}) to control $\|
  \nabla \bfu_{h,1} \|_{\OmegaO}$. This is precisely the role of
  $s_h(\bfu_h, \bfv_h) = (\nabla ( \bfu_{h,1} - \bfu_{h,2}),
  \nabla(\bfv_{h,1} - \bfv_{h,2}))_{\OmegaO}$.
\end{remark}

\section{\emph{A~priori} error estimate}
\label{p3:sec:a-priori}

The previous results on the interpolation errors, the Galerkin
orthogonality of the discretization and its stability enable us to
state the following \emph{a~priori} estimate for the error in the
discrete solution.
\begin{theorem}
  \label{p3:thm:a-priori}
  Let $k,l \geqslant 1$ and $(\alpha,\beta) = (\pm 1, 1)$.
  Assume that $(\bfu, p) \in
  [H^{k+1}(\Omega)]^d \times H^{l+1}(\Omega)$
  is a (weak) solution of the Stokes problem~\eqref{p3:eq:strongform}.
  Then the finite element solution $(\bfu_h, p_h)\in V_h^k \times
  Q_h^l$
  defined in~\eqref{p3:eq:stokes-olm} satisfies the following error
  estimate:
  \begin{equation}
    \tn ( \bfu - \bfu_h , p - p_h )\tn
    \leqslant C \left( h^{k} | \bfu
    |_{k+1,\Omega} + h^{l+1}|p|_{l+1,\Omega} \right).
    \label{p3:eq:a-priori}
  \end{equation}
\end{theorem}
\begin{proof}
  By applying the triangle inequality, we have
  \begin{align*}
    \tn (\bfu  - \bfu_h, p - p_h ) \tn
    &\leqslant
    \tn (\bfu  - \pi_h \bfu, p - \pi_h p ) \tn
    +
    \tn (\pi_h \bfu  - \bfu_h, \pi_h p - p_h ) \tnast.
  \end{align*}
  Given the interpolation estimate~\eqref{p3:eq:interpest-vp}, it is
  suffices to bound the discretization error $(\pi_h \bfu -
  \bfu_h,\pi_h p - p)$.  By Theorem~\ref{p3:thm:stability}, there
  exists a $(\bfv_h, p_h)$ such that $\tn (\bfv_h, p_h) \tnast \leqslant 1$
  and
  \begin{align}
    c \tn ( \pi_h \bfu  - \bfu_h , \pi_h p - p_h ) \tnast
    &\leqslant A_h(\bfu_h - \pi_h \bfu,p - \pi_h p; \bfv_h, q_h).
    \nonumber
  \end{align}
  Thus, the Galerkin orthogonality~\eqref{p3:eq:galerkin-ortho}
  and the definition of $A_h$ give
  \begin{align*}
    c \tn ( \pi_h \bfu  - \bfu_h , \pi_h p - p_h ) \tnast
    &\leqslant A_h(\bfu - \pi_h \bfu, p - \pi_h p; \bfv_h, q_h)
    \\
    &= \underbrace{a_h(\bfu - \pi_h \bfu, \bfv_h) + s_h(\bfu -
      \pi_h \bfu ,
    \bfv_h)}_{I}
    \\
    &\quad+
    \underbrace{b_h(\bfv_h, p - \pi_h p)}_{II}
    +
    \underbrace{b_h(\bfu - \pi_h \bfu,  q_h)}_{III}
    \\
    &\quad \underbrace{- S_h(\bfu - \pi_h \bfu;p - \pi_h
    p,\bfv_h,q_h)}_{IV}.
  \end{align*}
  To estimate the first term, we use \eqref{p3:eq:a_h-stable-2}
  and the interpolation estimate~\eqref{p3:eq:interpest-v}
  to obtain
  \begin{align*}
    |I| &\leqslant C \tn \bfu - \pi_h \bfu \tnast \tn \bfv_h \tnast
    \leqslant C h^{k} | \bfu |_{k+1, \Omega} \tn \bfv_h \tnast
    \leqslant C h^{k} | \bfu |_{k+1, \Omega}.
  \end{align*}
  Recalling that $\meanvalue{p} = p_{2}$, the second term $II$
  can be estimated by
  \begin{align*}
    | II |
    &=
    |(\nabla \cdot \bfv, p - \pi_h p)_{\OmcupOm} +
    (\bfn \cdot\jump{\bfv}, \meanvalue{p - \pi_h p})_{\Gamma}|
    \\
    &\leqslant
    \tn \bfv \tn \, \tn p - \pi_h p\tn
    + \ifnorm{\jump{\bfv}}{1/2}
    \ifnorm{p_{2} - \pi_{h,2} p_{2} }{-1/2}
    \\
    &\leqslant
    \tn \bfv \tn \, \tn p - \pi_h p\tn +
    C \ifnorm{\jump{\bfv}}{1/2}
    (\| p_{2} - \pi_{h,2} p_2 \|_{\Omega_2}
    + h\| \nabla (p_{2} - \pi_{h,2} p_2) \|_{\Omega_2})
    \\
    &\leqslant C
    \tn \bfv \tn h^{l+1} | p |_{l+1,\Omega}
    \leqslant C h^{l+1} | p |_{l+1,\Omega},
  \end{align*}
  where we used the trace inequality~\eqref{p3:eq:trace-inequality} in
  the penultimate step and
  finally~\eqref{p3:eq:inverse-estimates-for-triangles} in combination
  with the interpolation estimate~\eqref{p3:eq:interpest-p}. The third
  term can be treated similarly to obtain
  \begin{align*}
    | III |
    &\leqslant C h^{k+1} | \bfu |_{k+1,\Omega} \tn q_h \tnast
    \leqslant C h^{k+1} | \bfu |_{k+1,\Omega}.
  \end{align*}
  For the last term $IV$, several applications of Cauchy--Schwarz yield
  \begin{align*}
    |IV| &=
    \delta
    |\sum_{T \in \mcupm} h_T^2
    (-\Delta(\bfu - \pi_h \bfu), + \nabla ( p - \pi_h p),
    \pm\Delta \bfv_h + \nabla q_h)_T |
    \\
    &\leqslant
    C
    \bigl(
    \sum_{T \in \mcupm}
    h_T^2 \| \Delta(\bfu - \pi_h \bfu) \|_{\mcupm}^2
    +
    h_T^2 \| \nabla(p - \pi_h p) \|_{\mcupm}^2
    \bigr)^{1/2}
    \\
    &\phantom{\leqslant}\;\times
    \big(
    \sum_{T \in \mcupm}
    h_T^2 \| \Delta\bfv_h \|_{\mcupm}^2
    +
    h_T^2 \| \nabla q_h  \|_{\mcupm}^2
    \big)^{1/2}
  \end{align*}
  After applying the inverse
  inequalities~\eqref{p3:eq:inverse-estimate-delta-grad}
  and~\eqref{p3:eq:inverse-estimate-grad-l2} to the second
  term, the interpolation estimate~\eqref{p3:eq:interpest0-cutmes}
  implies
  \begin{align*}
    | IV |
    &\leqslant C ( h^{k} | \bfu |_{k+1,\Omega} + h^{l+1} | p |_{l+1,\Omega})
    \tn (\bfv_h, q_h ) \tn
    \leqslant C ( h^{k} | \bfu |_{k+1,\Omega} + h^{l+1} | p |_{l+1,\Omega}).
  \end{align*}
  Summing up, we obtain that
  \begin{align*}
    \tn (\bfu -\bfu_h,p - p_h) \tn
    &\leqslant
    \tn (\bfu -\pi_h \bfu,p - \pi_h p ) \tn +
    |I| + |II| + |III| + |IV|
    \\
    &\leqslant C ( h^{k} | \bfu |_{k+1,\Omega} + h^{l+1} | p
    |_{l,\Omega}),
  \end{align*}
  which concludes the proof.
  \qed
\end{proof}

\section{Condition number estimate}
\label{p3:sec:condition-number}

To conclude the analysis of the discretization presented in
Section~\ref{p3:sec:stokes-olm-method}, we here show that the
condition number of the associated stiffness matrix is uniformly
bounded by $C h^{-2}$ independently of the location of the overlapping
mesh $\mesh_2$. The proof of the condition number estimate follows the
approach of~\citet{Ern2006}.

Let $(\bfv_h, q_h) = \sum_{i=1}^N V_i \varphi_i$ where $\{ \varphi_i
\}_{i=1}^N$ is a basis for the element space $V_h \times Q_h$; $V =
\{V_i\}_i$ thus denotes the expansion coefficients of $(\bfv_h, q_h)$
in the given basis. Similarly, let $W$ label the expansion
coefficients of fields denoted $(\bfw_h, r_h)$. Denote the inner
product in $\R^N$ by $(V, W)_N = \sum_{i=1}^N V_i W_i$ and the
corresponding norm by $|V|_N^2 = (V,V)_N$. The stiffness matrix $\mcA$
associated with the form~\eqref{p3:eq:Ah-olm} is then defined as:
\begin{equation}
  \label{p3:eq:def:mcA}
  (\mcA V, W)_N = A_h( \bfv_h, q_h; \bfw_h, r_h)
  \quad \foralls \bfv_h,\bfw_h \in V_h, \foralls q_h,r_h \in Q_h .
\end{equation}

Since essential boundary conditions are applied for the velocity on
the whole of $\partial \Omega$, the discrete pressure is only
determined up to a constant, and so the matrix $\mcA$ has one zero
eigenvalue. Throughout the remainder of this section, we therefore
instead interpret $\mcA$ as $\mcA : \R^{N}/ \Kern(\mcA) \rightarrow
\Kern(\mcA)^\bot$, which is a bijective linear mapping by
definition. The condition number of the matrix $\mcA$ is then defined
by
\begin{equation}\label{p3:eq:defkappa}
  \kappa(\mcA) = | \mcA |_N  | \mcA^{-1} |_N,
\end{equation}
with
\begin{equation}
  | \mcA |_N = \sup_{x \in \widehat{\mathbf{R}}^N \setminus
  \boldsymbol{0}} \frac{|\mcA x|_N}{| x |_N }.
\end{equation}

We now state some estimates that will be needed in the derivation of
the estimate of the condition number. First, for a conforming,
quasi-uniform mesh $\mesh$ with mesh size $h$ and a finite element
space $\mcV_h$ defined on $\mesh$, it is well known that there are
positive constants $c_{M}$ and $C_{M}$ only depending on the
quasi-uniformness parameters and the polynomial order of $\mcV_h$ such
that the following equivalence holds:
\begin{align}
  \label{p3:eq:VRineq}
  c_{M} h^{d/2} | V |_N \leqslant \| v_h \| \leqslant C_{M} h^{d/2} | V |_N  \quad
  \foralls v_h \in \mcV_{h}.
\end{align}
A second ingredient is an inverse estimate and a Poincar\'e inequality
in the appropriate norms which we state in the following two lemmas.
\begin{lemma}
  There is a constant $C_I$ such that
  \begin{alignat}{2}
    \label{p3:eq:L2energyinv}
    \tn \bfv_h \tnast &\leqslant C_I h^{-1} \| \bfv_h \|_{\mcupm}, &\quad
    &\foralls \bfv_h \in V_h, \\
    \label{p3:eq:L2energyinv-2}
    \tn (\bfv_h,q_h) \tnast &\leqslant C_I h^{-1} \| (\bfv_h,q_h)
    \|_{\mcupm} &\quad &\foralls (\bfv_h,q_h)
    \in V_h \times Q_h.
  \end{alignat}
\end{lemma}
\begin{proof}
  By definition
  \begin{align*}
    \tn \bfv_h \tnast^2 = \| \nabla \bfv_h \|_{\mcupm}^2
    + \ifnorm{\langle \nablan \bfv_h \rangle}{-1/2}^2
    + \ifnorm{\jump{\bfv_h}}{1/2}^2 .
  \end{align*}
  Applying the standard inverse
  inequality~\eqref{p3:eq:inverse-estimates-for-triangles} on each
  mesh $\meshast_1$ and $\mesh_2$ separately, the first term is
  bounded by $C h^{-2}\| \bfv_h \|_{\mcupm}^2$.  A similar bound can
  be derived for the second term
  by~\eqref{p3:eq:inverse-estimate-for-facets}
  and~\eqref{p3:eq:inverse-estimates-for-triangles}.  The estimate for
  the remaining term follows in the similar manner by a combination
  of~\eqref{p3:eq:trace-inequality-for-FD} and
  again~\eqref{p3:eq:inverse-estimates-for-triangles}.

  The second estimate \eqref{p3:eq:L2energyinv-2} is an immediate
  consequence of the first and the fact that $1 = C \leqslant
  h^{-1}\diam(\Omega)$.
  \qed
\end{proof}
\begin{lemma} There is a constant $C$ such that
  \begin{align}
    \label{p3:eq:L2energy}
    \| \bfv \|_{\mcupm} &\leqslant C_P \tn \bfv \tnast,  \quad
    \foralls \bfv \in V_h.
  \end{align}
\end{lemma}
\begin{proof}
  By Proposition~\ref{p3:prop:l2-norm-on-olm}, we have
  \begin{align*}
    \| \bfv \|_{\mcupm}^2
    &=
    \| \bfv \|_{\meshast_1}^2 +
    \| \bfv \|_{\mesh_2}^2
    \leqslant
    \| \bfv \|_{\Omega_1}^2 + C( \| \nabla \bfv \|^2_{\meshast_1}
    + \| \bfv \|_{\Omega_2}^2)
    \\
    &\leqslant C( \tn \bfv \tnast^2 + \| \bfv \|_{\OmcupOm}^2) ,
  \end{align*}
  for all $\bfv \in V_h$. Hence, it remains to derive a bound for a
  $\| \bfv \|_{\OmcupOm}$.

  We use a duality argument to estimate $\| \bfv \|_{\OmcupOm}$.  Let
  $\bfphi \in H^2(\Omega)$ be the solution of the dual problem
  $-\Delta \bfphi = \bfv$ with boundary conditions $\bfphi = 0$ on
  $\partial \Omega$.  Multiplying with $\bfv_h$ and integrating by
  parts we get by using the Cauchy-Schwarz inequality and the trace
  inequality \eqref{p3:eq:trace-inequality-for-FD} for $\bfphi$:
  \begin{align*}
    \|\bfv\|^2_{\OmcupOm}
    &=
    \sum_{T \in \mesh_1} (\nabla \bfv, \nabla \bfphi)_{T \cap
    \Omega_1}
    + \sum_{T \in \mesh_2} (\nabla \bfv, \nabla \bfphi)_{T}
    -  ([\bfv], \bfn\cdot \nabla \bfphi)_{\Gamma}
    \\
    &\leqslant \left(
    h^{-1} \|\jump{\bfv}\|_{\Gamma}^2 + \sum_{T \in \mcupm} \| \nabla \bfv \|^2_T
    \right)^{\frac{1}{2}}
    \\
    &\qquad \times \left( \sum_{T \in \mcupm} \| \nabla \bfphi \|_T^2
    + \sum_{T \in \mesh_{0,\Gamma}}  C (\| \nabla \bfphi
    \|_T^2 + h^2 |\nabla \bfphi |_{1,T}^2 ) \right)^{\frac{1}{2}}
    \\
    &\leqslant \tn \bfv \tnast \, \| \bfphi \|_{2, \Omega}.
  \end{align*}
  Finally, using standard elliptic regularity~\citep{BrennerScott2008}
  $\|\bfphi \|_{2, \Omega} \leqslant \| \bfv \|_{\Omega}$ and division by
  $\|\bfv\|_{\Omega}$, the estimate follows.
  \qed
\end{proof}

\begin{theorem}
  \label{p3:thm:conditionnumber}
  The condition number of the stiffness matrix satisfies the estimate
  \begin{equation}
    \kappa(\mcA) \leqslant C h^{-2}.
    \label{p3:eq:conditionnumber}
  \end{equation}
\end{theorem}
\begin{proof}
  Recalling the definition of the condition number
  in~\eqref{p3:eq:defkappa}, the proof consists of deriving estimates
  of $|\mcA|_N$ and $|\mcA^{-1}|_N$.  In the following, we use the
  following well-known equivalence for the Euclidean norm $|x|_N$:
  \begin{equation*}
    |V|_N = \sup_{W\neq \boldsymbol{0}} \frac{(V,W)_N}{|W|_N}.
  \end{equation*}
  \paragraph{Estimate of $|\mcA|_N$.}
  We have by the definition of $\mcA$
  \begin{align}
    \label{p3:eq:cond11}
    | \mcA V|_N &= \sup_{W \neq \boldsymbol{0}} \frac{(\mcA V,W)_N}{|W|_N}
    \\
    \label{p3:eq:cond12}
    &=\sup_{W \neq \boldsymbol{0}} \dfrac{A_h(\bfv_h,q_h; \bfw_h,r_h)}{|W|_N}
    \\
    \label{p3:eq:cond13}
    &\leqslant C_A \sup_{W \neq \boldsymbol{0}} \dfrac{\tn ( \bfv_h,q_h) \tnast
    \cdot \tn(\bfw_h,r_h)\tnast}{|W|_N}
    \\
    \label{p3:eq:cond14}
    &\leqslant C_A C_I^2 \sup_{W \neq \boldsymbol{0}}  \dfrac{h^{-1} \| (\bfv_h,q_h)
    \|_{\mcupm}
    \cdot h^{-1} \|(\bfw_h,r_h)\|_{\mcupm}}{|W|_N}
    \\
    \label{p3:eq:cond15}
    &\leqslant C_A C_I^{2} C_M^2 | V |_N.
  \end{align}
  Here, inequality \eqref{p3:eq:cond13} follows by the continuity of
  $A_h$ with respect to the norm $\tn \cdot
  \tnast$,~\eqref{p3:eq:cond14} follows by using the inverse estimate
  \eqref{p3:eq:L2energyinv-2} twice, and finally~\eqref{p3:eq:cond15}
  results by applying~\eqref{p3:eq:VRineq}. Dividing by $| V |_N$ and
  taking the supremum over all $V \not = \mathbf{0}$, we conclude that
  \begin{align}\label{p3:condA}
    | \mcA |_N \leqslant C_A C_I^{2} C_M^2 .
  \end{align}
  \paragraph{Estimate of $|\mcA^{-1}|_N$.}
  Starting with~\eqref{p3:eq:VRineq} and sequentially using the
  Poincar\'e inequality~\eqref{p3:eq:L2energy} (extended to the
  product space), the inf-sup stability~\eqref{p3:eq:stability} of
  $A_h$ and finally the Poincar\'e inequality again, we arrive at the
  following chain of estimates:
  \begin{align*}
    h^{d/2} |V|_N
    &\leqslant c_M^{-1} \| (\bfv_h,q_h) \|_{\mcupm}
    \leqslant c_M^{-1} C_P \tn (\bfv_h,q_h) \tnast \\
    &\leqslant c_M^{-1} C_P c_A^{-1} \frac{A(\bfv_h,q_h;\bfw_h,r_h)}
      {\tn ( \bfw_h, r_h) \tnast}
      = c_M^{-1} C_P c_A^{-1} \frac{(\mcA V, W)_N}{\tn ( \bfw_h, r_h)
    \tnast}
    \\
    &\leqslant c_M^{-1} C_P c_A^{-1}|\mcA V |_N \, \frac{|W|_N}{\tn (
    \bfw_h, r_h) \tnast}
    \\
    &\leqslant c_M^{-2} C_P c_A^{-1}|\mcA V |_N \, h^{-d/2}
    \frac{\|(\bfw_h, r_h)\|_{\mcupm}}{\tn ( \bfw_h, r_h) \tnast}
    \\
    &\leqslant c_M^{-2} C_P^2 c_A^{-1}|\mcA V |_N \,h^{-d/2}
  \end{align*}
  Setting $V = \mcA^{-1} W$ we have $W=\mcA V$ and the inequality now reads
  $|\mcA^{-1} W |_N \leqslant  c_M^{-2} C_P^2 c_A^{-1} h^{-d} | W |_N$
  for all $W \not = 0$, or
  in other words
  \begin{align}\label{p3:condB}
    |\mcA^{-1} |_N \leqslant c_M^{-2} C_P^2 c_A^{-1} h^{-d}.
  \end{align}
  Finally, combining~\eqref{p3:condA} and~\eqref{p3:condB}
  yields~\eqref{p3:eq:conditionnumber}.
  \qed
\end{proof}

\section{Numerical examples}
\label{p3:sec:num-examples}

We conclude this paper with three numerical tests; all in three
spatial dimensions. The numerical experiments were carried out using
the \rm{DOLFIN-OLM} library
(\url{http://launchpad.net/dolfin-olm}). We first corroborate the
theoretical \emph{a~priori} error estimate with a convergence
experiment. Second, we numerically examine how the location of the
overlapping mesh in relation to the background mesh affects the
condition number. Finally, we demonstrate how the method presented and
the features provided by \rm{DOLFIN-OLM} can be applied to the flow
around an airfoil in a three dimensional channel. The experiments were
performed with $V_h = V_h^1$, $Q_h = Q_h^1$ and $\beta = 1$.

\subsection{Convergence test}
\label{p3:ssec:convergence-test}
We let $\Omega = [0, 1]^3$ and choose the overlapping domain
$\Omega_1$ to be a rotation along the $y$-axis of the domain
$\widetilde{\Omega}_1 = [0.3331, 0.6669]^3$ as illustrated
in~Figure~\ref{p3:fig:convergence:overlapping-meshes}.  To examine the
convergence of the methods, we apply the method of manufactured
solutions: let $\bfu(x,y,z) = (\sin(\pi y) \sin(\pi z), 0, 0)$ and
$p(x,y,z) = \cos(\pi x) + 1$. The right-hand side $\bff$ is defined
accordingly so that~\eqref{p3:eq:strong-stress} is satisfied, and the
corresponding Dirichlet boundary conditions are strongly imposed on
the entire boundary $\partial \Omega$.  The numerical approximation
corresponding to~\eqref{p3:eq:stokes-olm} is then computed on a
sequence of overlapping uniform meshes
$\{(\mesh_0^N,\mesh_2^N)\}_{N=0}^{5}$.  The mesh size of the initial
meshes $\mesh_0^0$ and $\mesh_2^0$ is $h_{\max} \approx 1/3$ and each
subsequent pair $(\mesh_0^N,\mesh_2^N)$ is generated from the previous
one by uniformly refining each mesh.
\begin{figure}
  \begin{center}
    \includegraphics[width=0.59\textwidth]{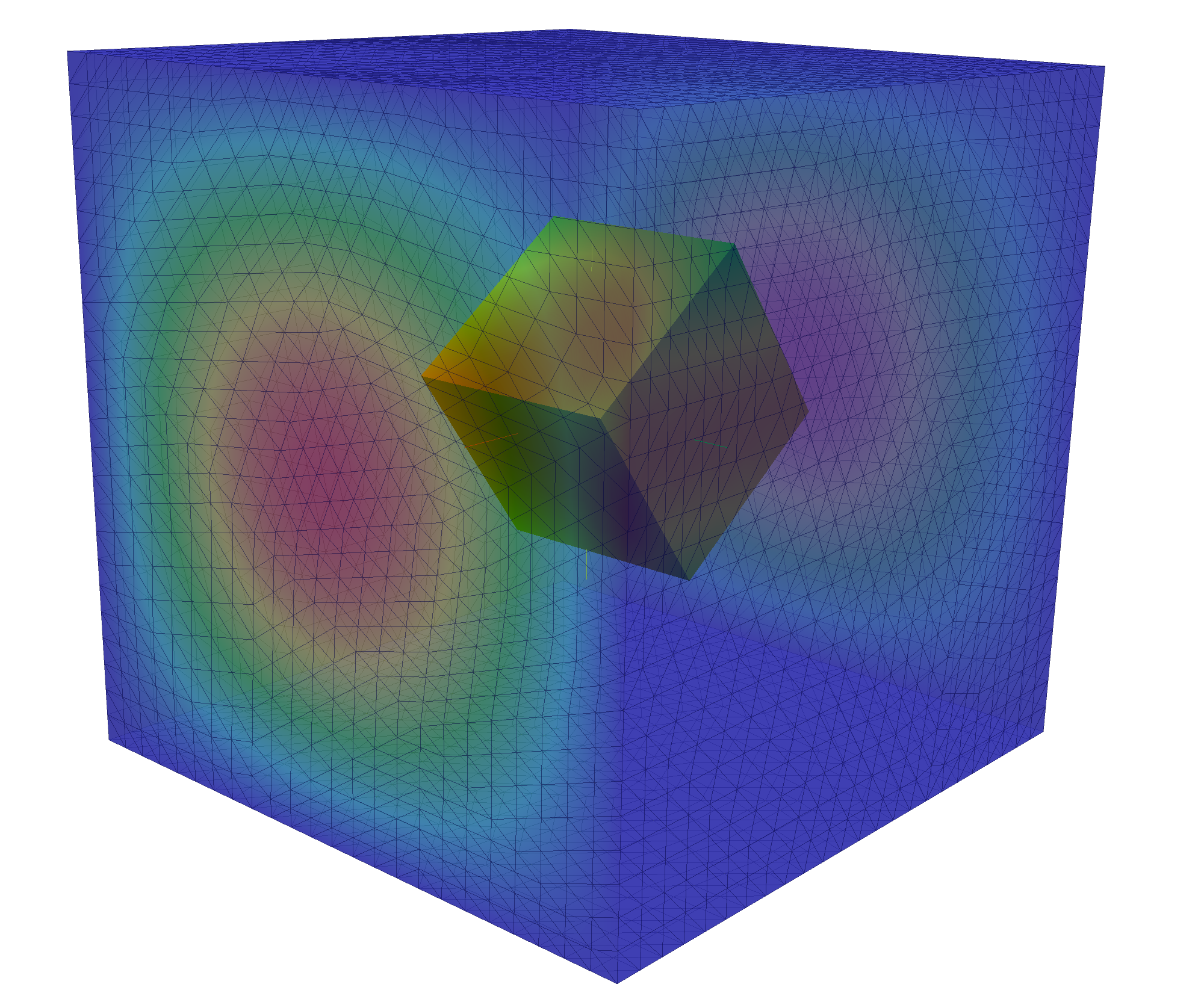}
    \caption{Mesh configuration and magnitude of velocity
      approximation corresponding to this mesh resolution for the
      numerical convergence test.}
    \label{p3:fig:convergence:overlapping-meshes}
  \end{center}
\end{figure}

The stabilization parameters were set to $\gamma=10$ and $\delta =
0.05$.  To solve the linear system of equations, we apply a
transpose-free quasi-minimal residual solver with an algebraic
multigrid preconditioner, filtering out the constant pressure mode in
the iterative solver. The solves converged to a tolerance of $10^{-8}$
in between $34$ and $47$ iterations.
\begin{figure}
  \begin{center}
    \includegraphics[width=0.49\textwidth]{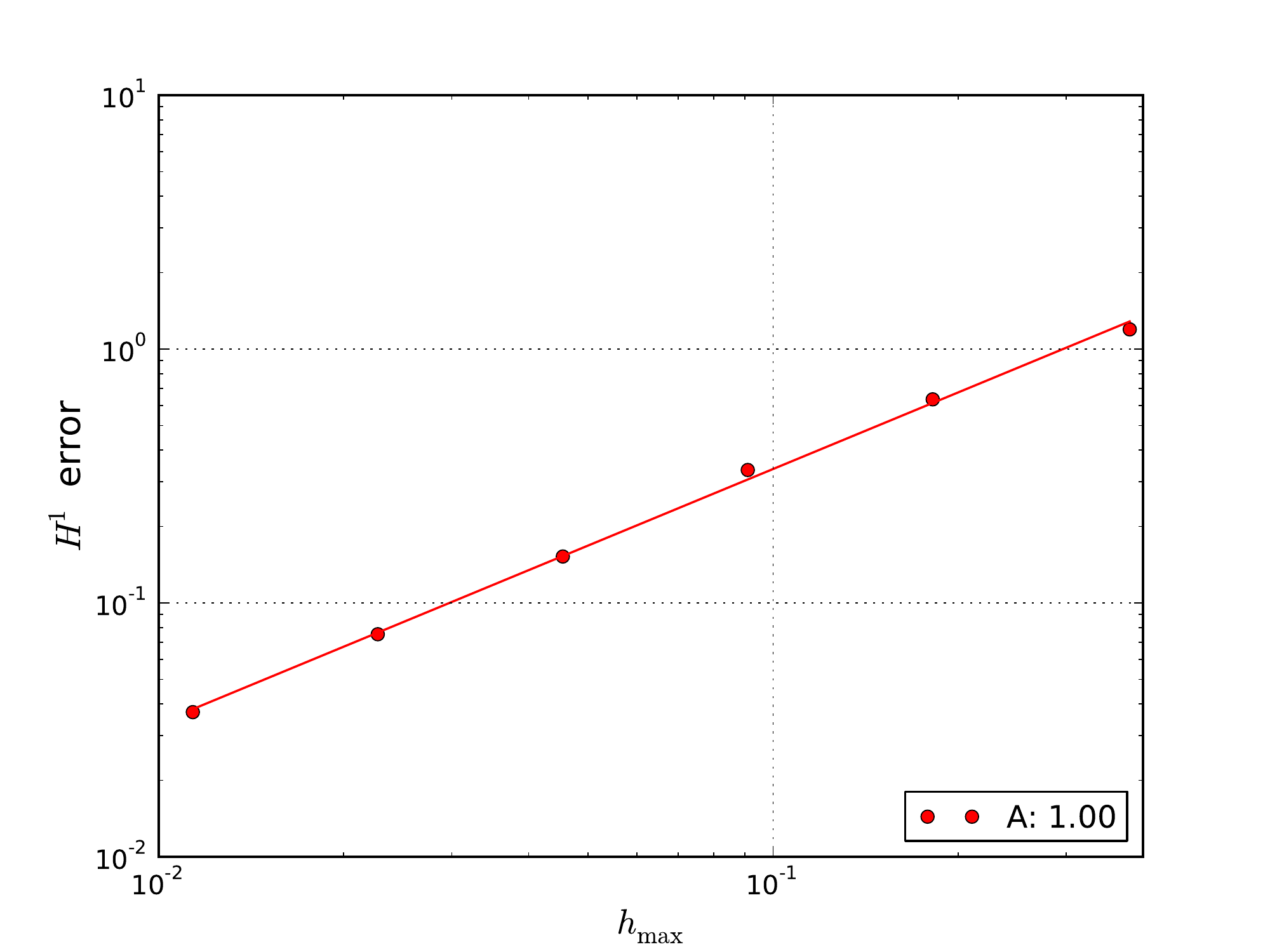} 
    \includegraphics[width=0.49\textwidth]{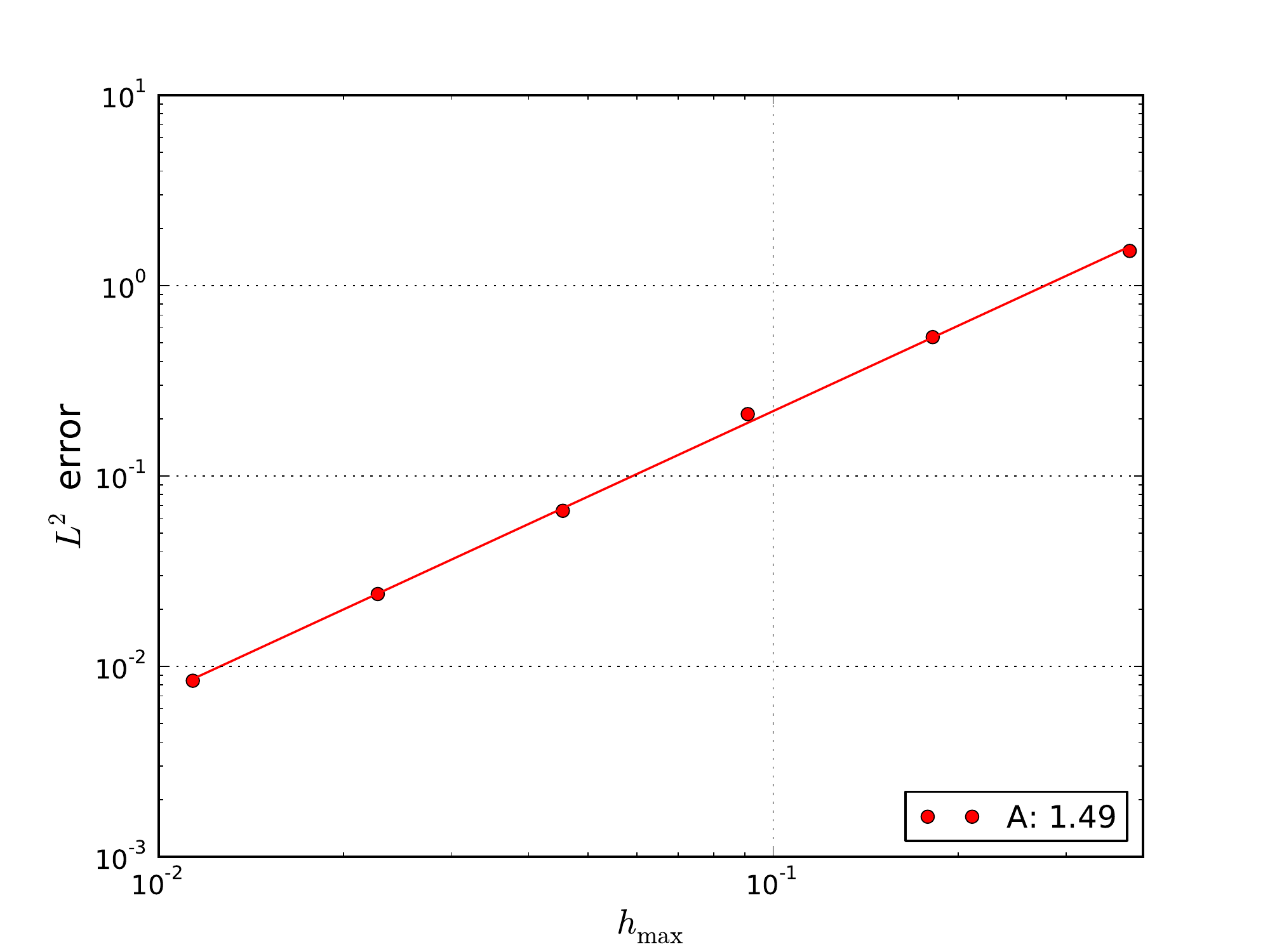}
    \caption{Convergence of the stabilized Nitsche overlapping mesh
      method for $V_h^1\times Q_h^1$. The legend gives the fitted
      slope for each configuration. Top: $H^1$-error $||\bfu -
      \bfu_h||_{1, \mcupm}$ for the velocity versus maximal element
      diameter $h_{\max}$. Bottom: $L^2$-error $||p - p_h||_{\mcupm}$
      for the pressure versus $h_{\max}$.}
    \label{p3:fig:convergence:p1xp1-olm}
  \end{center}
\end{figure}

The resulting errors for the sequence of refined meshes are given in
Figure~\ref{p3:fig:convergence:p1xp1-olm}.
Theorem~\ref{p3:thm:a-priori} predicts first order convergence for
both the $H^1$-norm of the velocity error and the $L^2$-norm of the
pressure error. This is also (at least) observed in the numerical
results, thus corroborating the theoretical estimate. We additionally
observe that the pressure approximation seems to converge at a higher
order ($\approx 1.5$) for the range of meshes investigated here.

\subsection{Condition number tests}
\label{p3:ssec:condition-tests}

The next numerical example demonstrates that the condition number of
the matrix $\mcA$ defined by~\eqref{p3:eq:def:mcA} is bounded
independently of the position of the overlapping mesh relative to the
background mesh. This bound can be attributed to the
term~\eqref{p3:eq:sh-olm} defined on the overlap region: we also
illustrate that the condition number seems unbounded if this term is
not included.

Let $\Omega = [0, 1]^3$ be tessellated by $N^3$ cubes, each cube
divided into six tetrahedra. Define the overlapping domain by
$\Omega_2 = \Omega_2(l) = [l, 1-l]^3$ for a parameter $l \in (0, 1)$,
and tessellate $\Omega_2$ in the same manner as $\Omega$, but with
$M^3$ cubes. We will consider two cases of mesh sizes: (i) $N = 5, M =
3$, and (ii) $N = 10, M = 6$. For both cases, $l$ approaching $0.2$
from above corresponds to a limiting case: a few of the degrees of
freedom in the overlapped mesh will only receive a contribution from
the integrals defined over the overlap region while the cell and
interface contributions vanish. Here, we therefore examine $l \in
(0.2, 0.21]$. For each $l$ and choice of $(N, M)$, we compute the
condition number $\kappa$ of the corresponding matrix $\mcA$ as the
ratio of the absolute values of the largest (in modulus) eigenvalue
and the smallest (in modulus), nonzero, eigenvalue.

Let $\delta = 0.05$ and $\gamma = 10.0$ as before. The resulting
condition numbers, scaled by the square mesh size of the overlapping
mesh, are given in Table~\ref{p3:tab:condition}. We observe that for
each $l$, the difference in the condition number between the two mesh
sizes is small, as expected in view of
Theorem~\ref{p3:thm:conditionnumber}. Moreover, for both the case $N =
5$ and $N = 10$, the scaled condition numbers seem clearly bounded as
$l \rightarrow 0.2$. In contrast, the scaled condition number grows
significantly as $l \rightarrow 0.2$ when the overlap integrals $s_h$
are not included.
\begin{table}[ht]
  \centering
  \begin{tabular}{l|rrrrr}
    \toprule
    Case &  $l = 0.21$ &  $0.201$ &  $0.2001$ &  $0.20001$ & $0.200001$  \\
    \midrule
    $N = 5$  ($h \approx 0.33$) & $1076$ & $1207$ & $1220$ & $1222$ & $1222$ \\
    $N = 10$ ($h \approx 0.17$) & $955$ & $1149$& $1170$ & $1173$ & $1174$ \\
    \midrule
    $N = 5$ without $s_h$ & $583$ & $643$ & $958$ & $9715$ & $110636$ \\
    \bottomrule
  \end{tabular}
  \caption{Scaled condition numbers $\kappa h^2$ where $h$ is the
    minimal cell diameter of the overlapping mesh. Each column to
    corresponds one overlapping domain $\Omega(l)$ for $l = 0.21$, $l =
    0.201$ etc. The bottom row corresponds to the matrix induced by the
    form $A_h$ \emph{without} the overlap integrals $s_h$.}
    \label{p3:tab:condition}
\end{table}

\subsection{Flow around an airfoil in a channel}
\label{p3:ssec:illust-example}

Finally, we illustrate that the method and technology developed here
can be successfully applied to more realistic numerical
simulations. In particular, we consider the flow around an airfoil in
a channel.

As $\mesh_{2}$ we take a tetrahedral mesh approximation of a sphere
surrounding an airfoil, discretizing the boundary layer around the
airfoil with a higher mesh resolution than the remainder of the
domain. As $\mesh_0$ we take a tessellation of $[-3, 3]^3 \backslash
O$ where the interior domain $O$ is contained in the convex hull of
$\mesh_{2}$. (In total, $\mesh_0 \cup \mesh_2$ is a mesh of a $[-3,
3]^3$ box with an airfoil removed.)

The stabilization parameters and choice of finite element spaces are
as before: $V_h^1 \times Q_h^1$, $\beta = 1$, $\delta = 0.05$ and
$\gamma = 10$. We enforce a parabolic velocity profile at the inflow
boundary, no slip conditions for the velocity on the outer walls and
the airfoil boundary, and stress-free boundary conditions at the
outflow boundary and take $\bff = \mathbf{0}$.

The flow patterns around the airfoil can now be studied for instance
for different angles of attack by rotating $\mesh_2$ while keeping
$\mesh_0$ fixed. The velocity and pressure approximations for a series
of mesh pairs, in which $\mesh_2$ was first rotated around the
$z$-axis with angle $\theta$ and then around the $y$-axis with the
same angle, are visualized in Figures~\ref{p3:fig:fun-experiment}
and~\ref{p3:fig:fun-experiment-streamlines}. We especially note the
smooth transition of the solution from $\mesh_0$ to $\mesh_2$; the
interface is not visible.
\begin{figure}
  \begin{center}
    \includegraphics[width=0.48\textwidth]{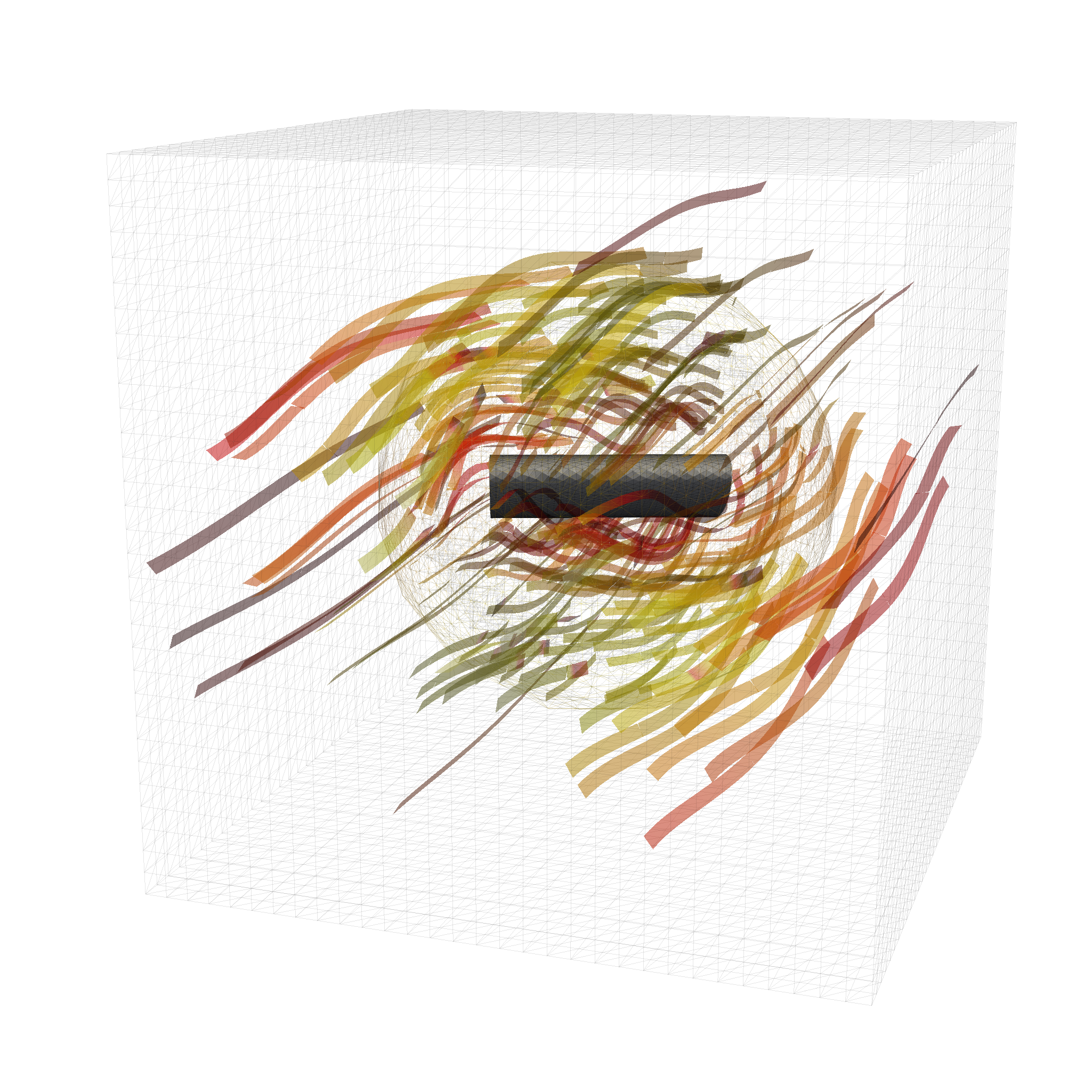}
    \includegraphics[width=0.48\textwidth]{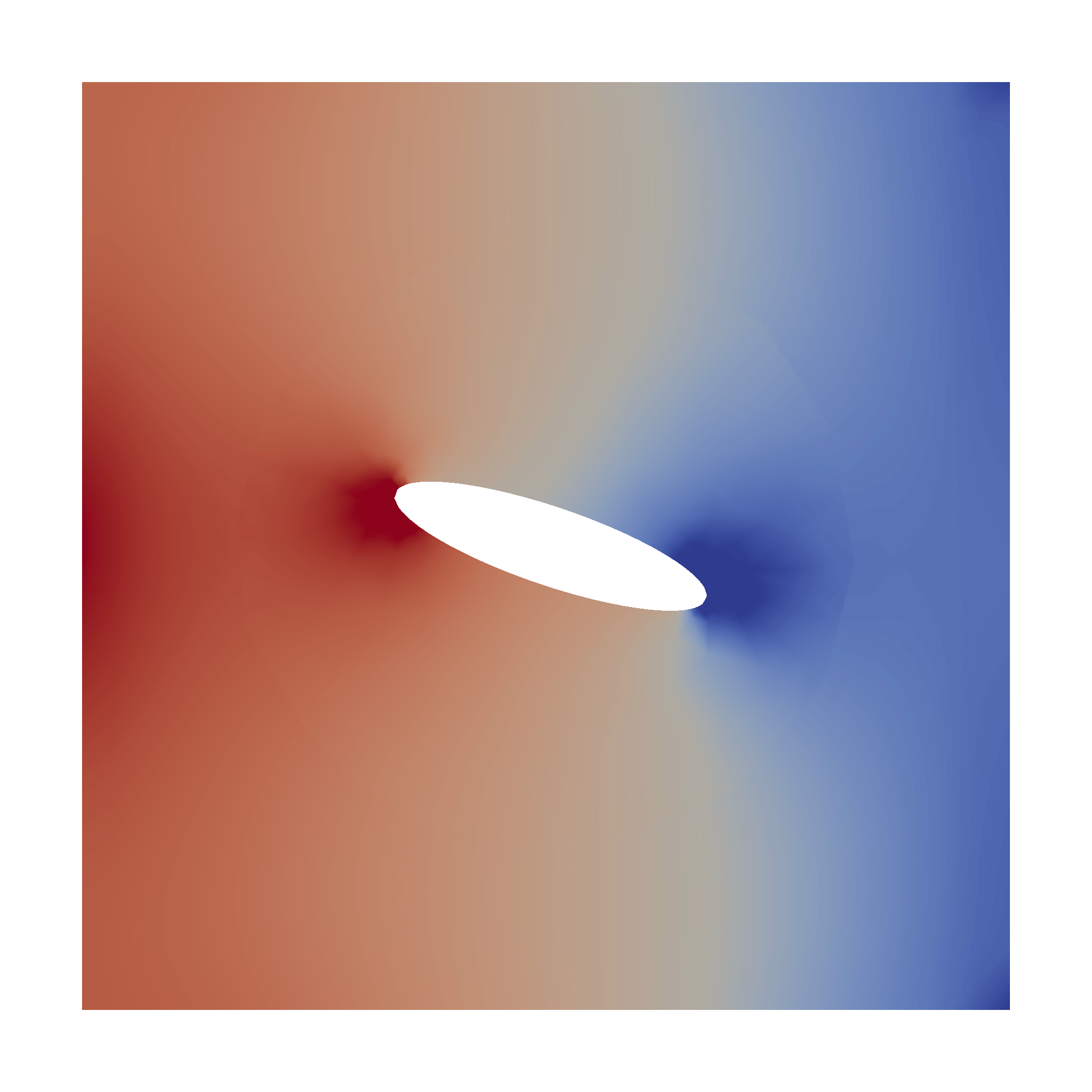}
    \\
    \vspace{-2em}
    \includegraphics[width=0.48\textwidth]{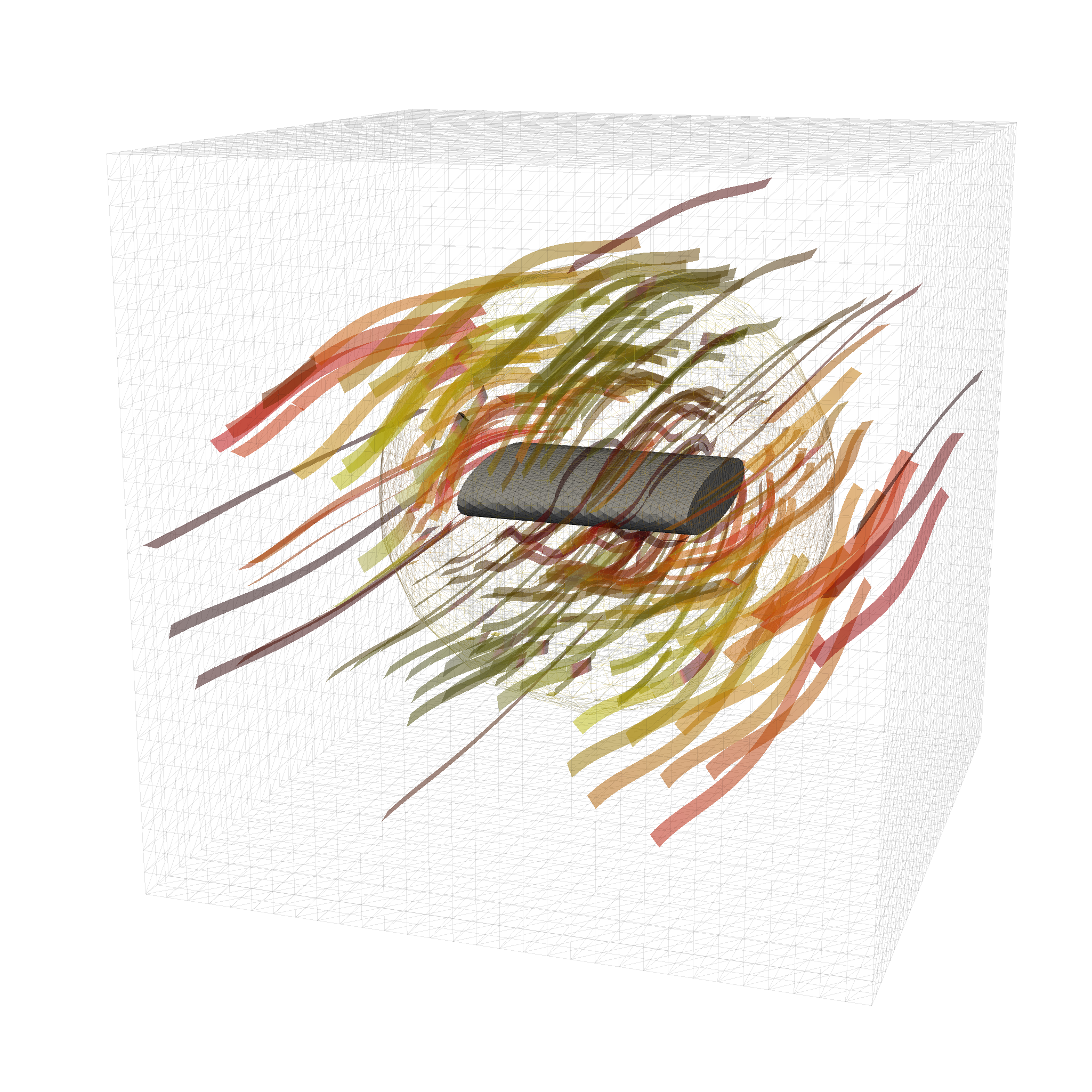}
    \includegraphics[width=0.48\textwidth]{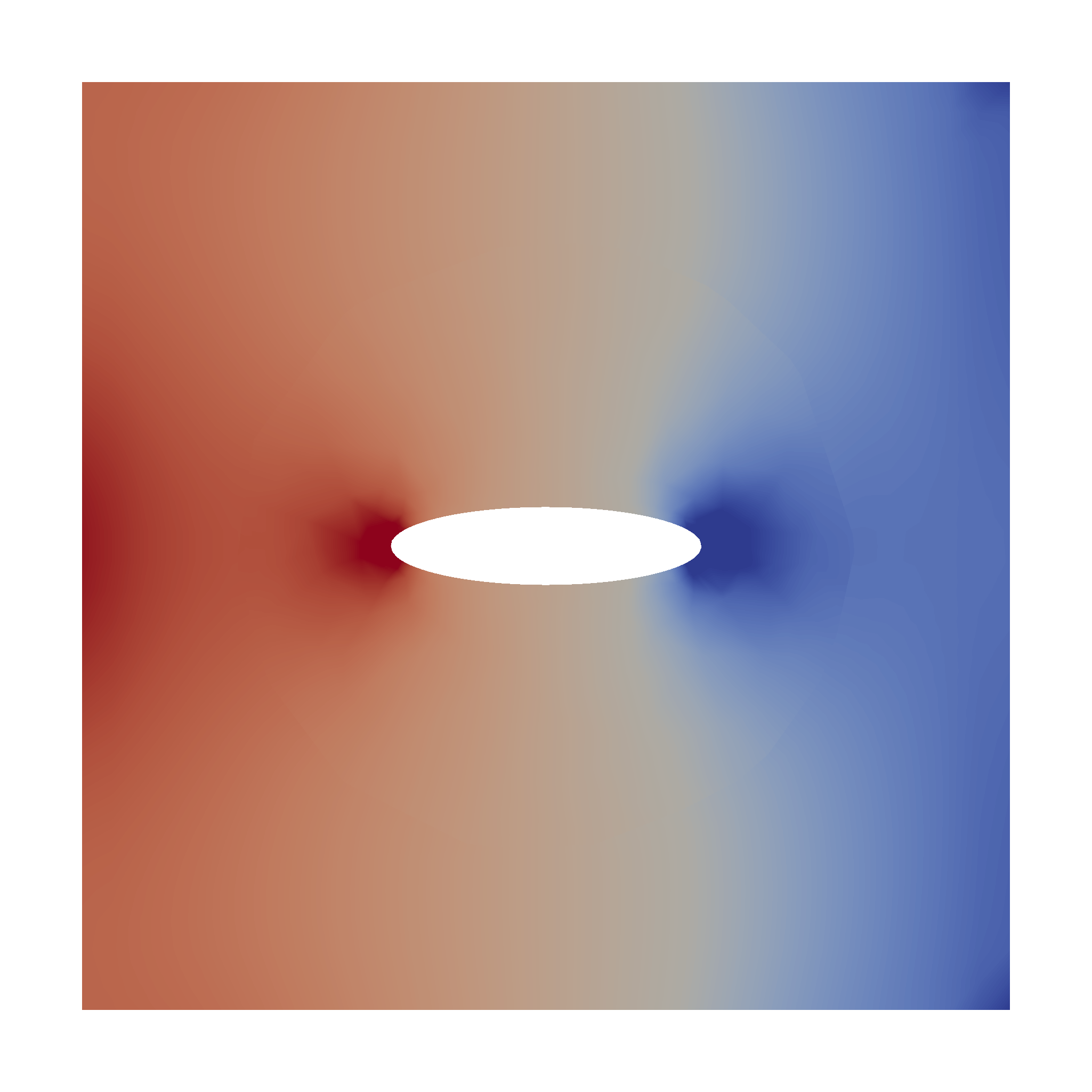}
    \\
    \vspace{-2em}
    \includegraphics[width=0.48\textwidth]{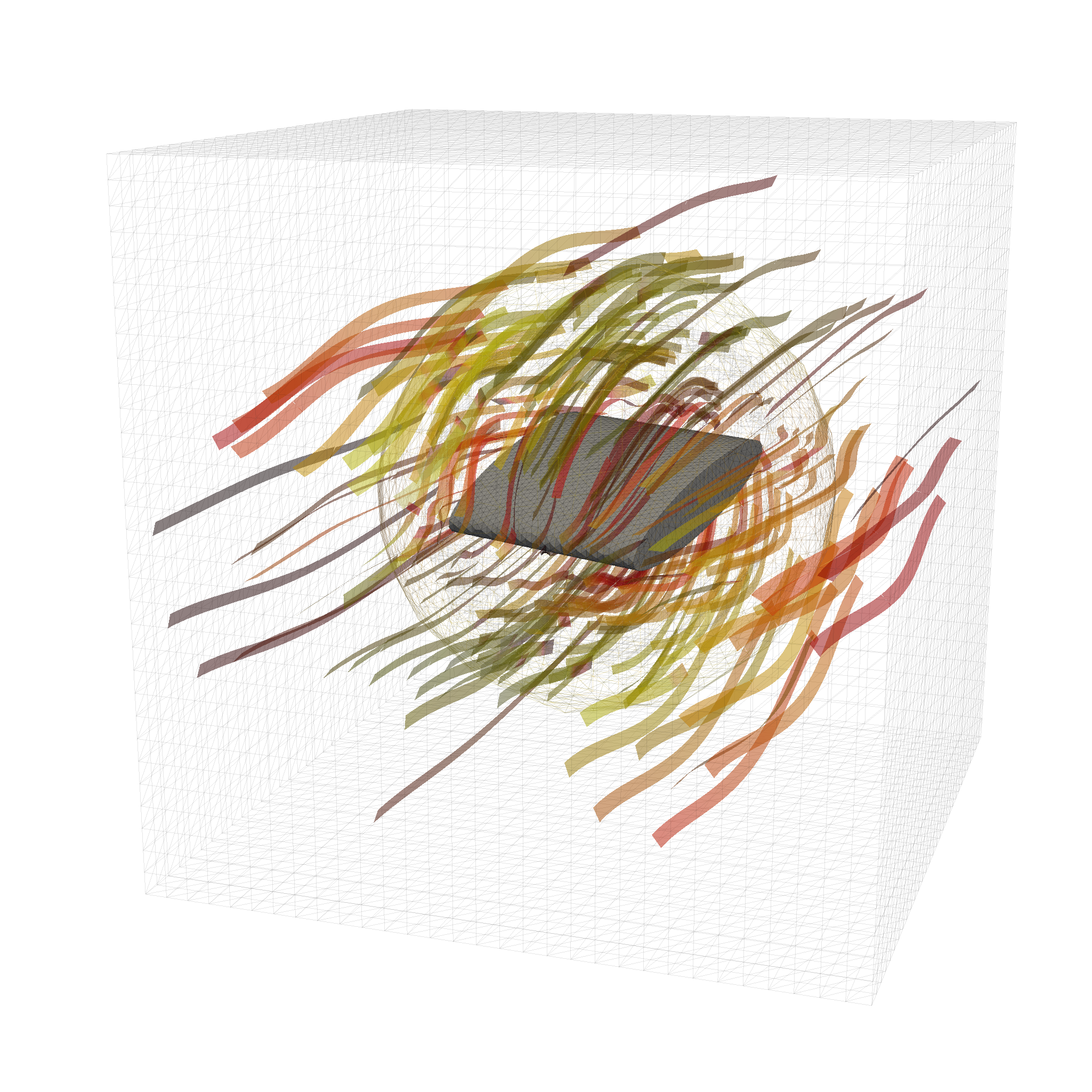}
    \includegraphics[width=0.48\textwidth]{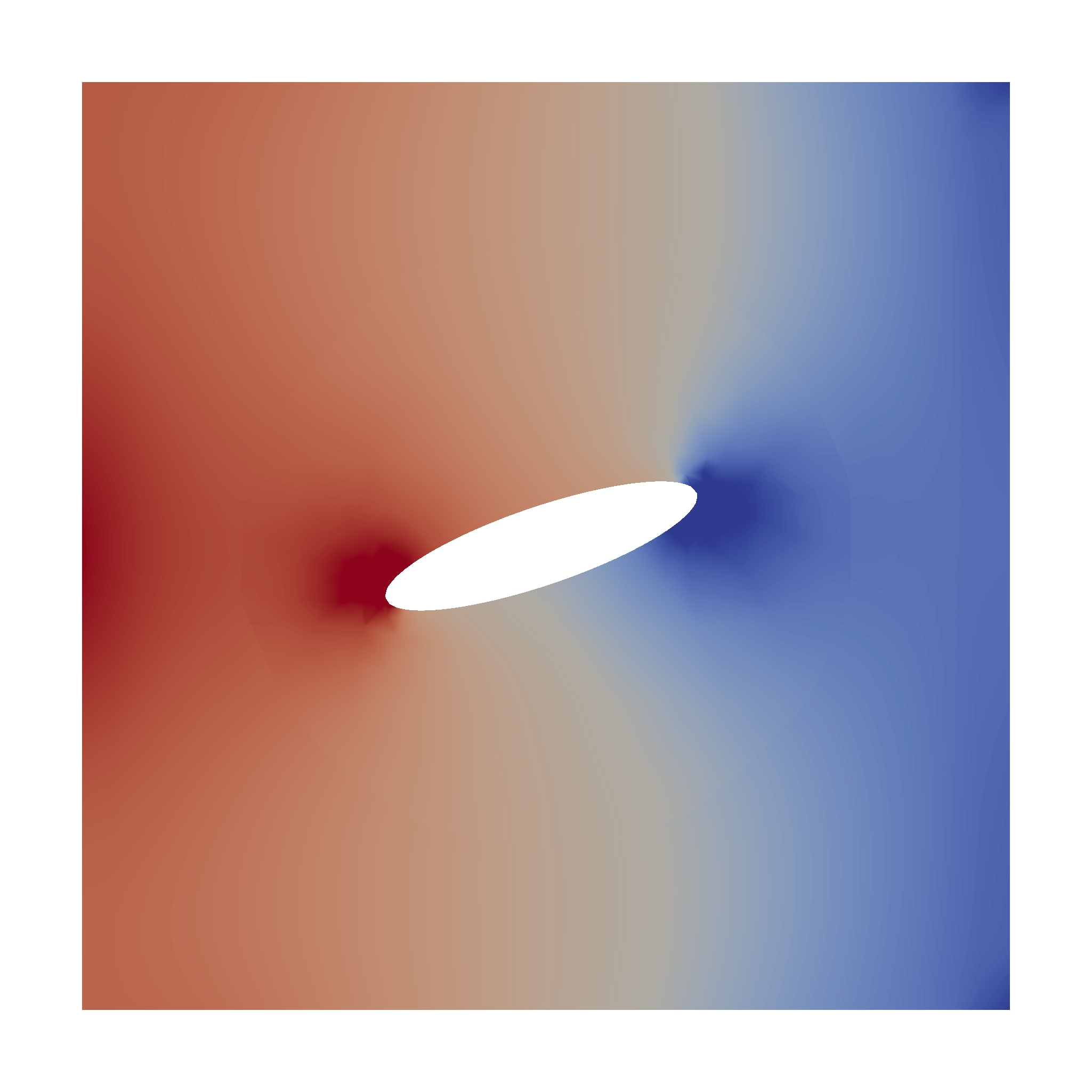}
    \\
    \vspace{-2em}
  \end{center}
  \caption{Velocity streamlines (left) and pressures in the plane
    defined by the $z$-axis (right) for different angles of attack
    $\theta$: from top to bottom: $\theta = 20^{\circ}, 0^{\circ},
    -20^{\circ}$.}
  \label{p3:fig:fun-experiment}
\end{figure}

\begin{figure}
  \begin{center}
    \includegraphics[width=0.74\textwidth]{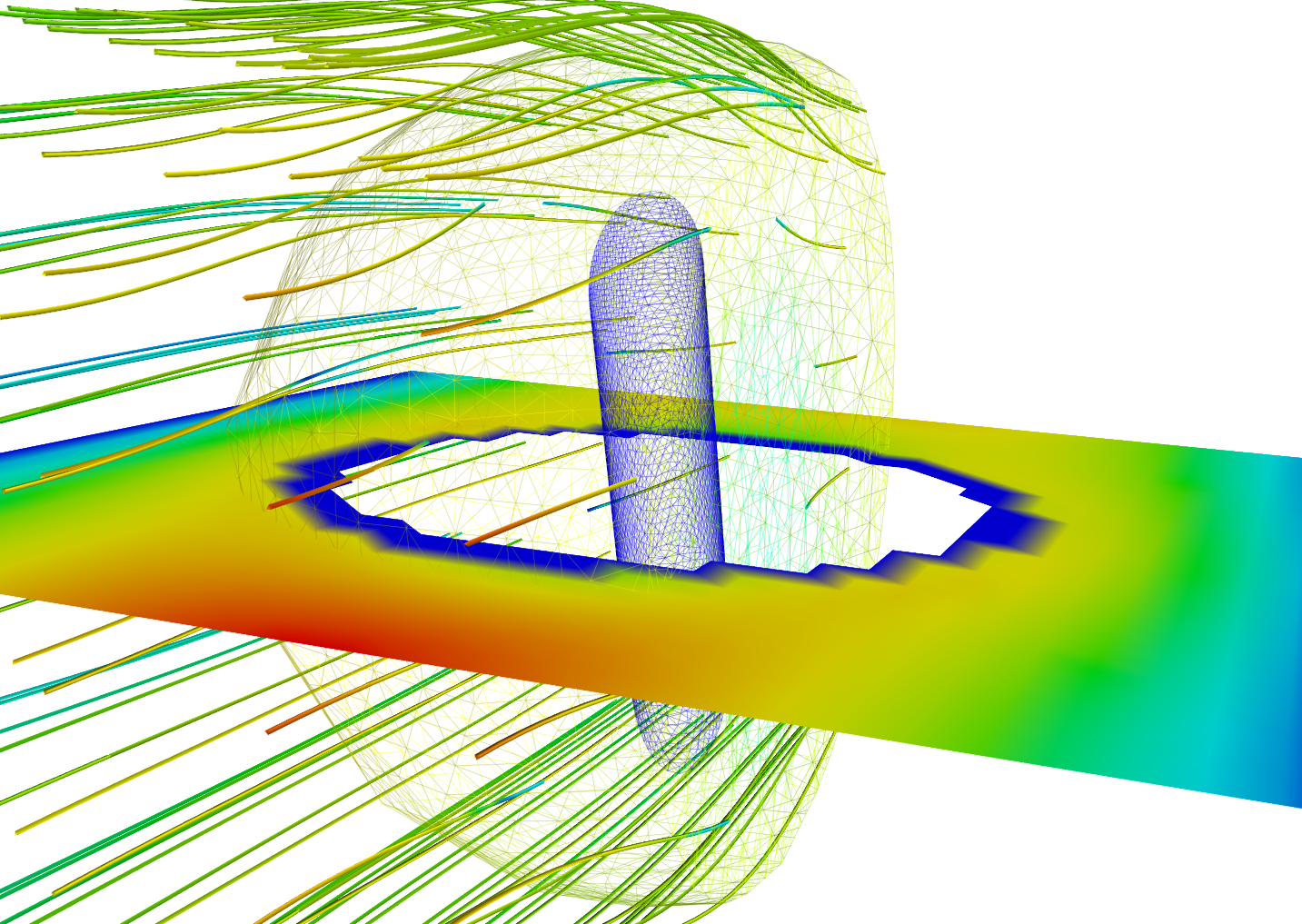}
    \\
    \vspace{0.5cm}
    \includegraphics[width=0.74\textwidth]{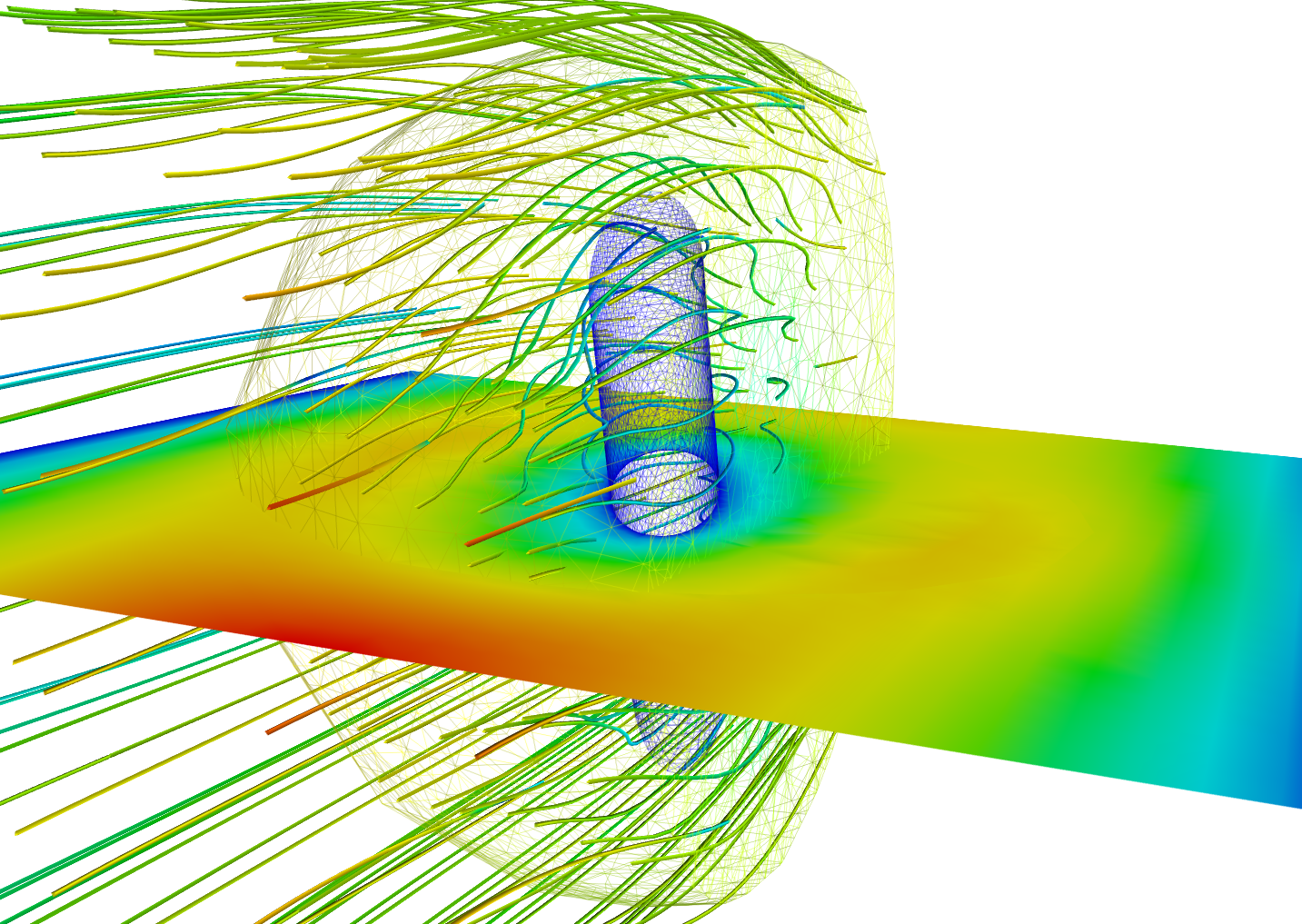}
    \caption{Velocity streamlines and magnitude in a cross-section for
      angle of attack $\theta = -20^{\circ}$. Top: The solution on the
      background mesh $\meshast_1$. It can clearly be seen that the
      streamlines stop close to the interface zone. Bottom: The
      solution on the overlapping mesh $\mesh_2$ superimposed on the
      solution on the background mesh $\meshast_1$, indicating a
      smooth transition of the solution from $\meshast_1$ to
      $\mesh_2$.}
    \label{p3:fig:fun-experiment-streamlines}
  \end{center}
\end{figure}

\section{Conclusions}

We have presented and analyzed a general class of stabilized finite
element methods for the Stokes problem posed on overlapping,
non-matching meshes. The interface conditions are enforced using
Nitsche's method, yielding a provably optimally convergent method. In
addition, a least-squares term on the overlap region is included to
control the condition number of the stiffness matrix.

The theoretical results have been verified numerically for a test
problem consisting of Stokes flow through a channel described by a
sphere superimposed on a tetrahedral mesh of the unit cube. We have
further verified that the condition number of the stiffness matrix
remains bounded, independently of the location of the overlapping mesh
relative to the background mesh. Finally, we have demonstrated the
applicability of the proposed method and our implementation to the
simulation of Stokes flow around a three-dimensional airfoil with a
fitted mesh superimposed on a non-matching background mesh.

A natural extension of the work presented in this paper is to the
simulation of fluid--structure interaction problems where a fluid mesh
of the surroundings of an elastic body is superimposed on a background
fluid mesh as in Figure~\ref{p3:fig:stokes_overlapping}. The
fluid--structure interaction may then be handled via a standard
arbitrary Lagrangian--Eulerian (ALE) approach on the overlapping mesh,
while the Nitsche overlapping mesh method analyzed in this paper is
used to enforce the interface conditions across the fluid--fluid
boundary. We explore this technique further in ongoing work.

While all software used in this work is available as free/open-source,
additional work is required to create interfaces and documentation
that make the software useful to a general audience. This issue will
be addressed in the near future, with the goal to provide an
easy-to-use programming environment for overlapping mesh methods as
part of the FEniCS software suite.

\begin{acknowledgements}
This work is supported by an Outstanding Young Investigator grant from
the Research Council of Norway, NFR 180450. This work is also
supported by a Center of Excellence grant from the Research Council of
Norway to the Center for Biomedical Computing at Simula Research
Laboratory.
\end{acknowledgements}
\clearpage

\bibliographystyle{plainnat}
\bibliography{bibliography}

\end{document}